\documentclass{article}

\usepackage{hyperref}
\usepackage{a4wide}
\usepackage{amsfonts,amssymb,amsmath,amsthm}
\usepackage{graphicx}
\usepackage[latin1]{inputenc}
\usepackage{color}
\usepackage{epstopdf}
\usepackage{subfigure}
\usepackage{mathabx} 
\numberwithin{equation}{section}
\usepackage{booktabs}
\usepackage{float}
\usepackage{caption}
\newcommand{\half}{\frac{1}{2}}
\newcommand{\dud}[1]{\frac{\partial v}{\partial #1}}
\newcommand{\dudd}[1]{\frac{\partial^2v}{\partial #1^2}}
\newcommand{\duddm}[2]{\frac{\partial^2v}{\partial #1 \partial #2}}
\newcommand{\dd}{\mathrm{d}}

\newtheorem{theorem}{Theorem}

\newtheorem{remark}[theorem]{Remark}

\def\R{\mathbb{R}}
\def\C{\mathbb{C}}

\author{Karel~J.~in 't Hout\footnote{Department of Mathematics,
		University of Antwerp, Middelheimlaan 1, B-2020 Antwerp, Belgium.
		\mbox{Email}: \texttt{\{karel.inthout,pieter.lamotte\}@uantwerpen.be}.}
	~and Pieter Lamotte\footnotemark[\value{footnote}]
}

\title{Efficient numerical valuation of European options\\ 
under the two-asset Kou jump-diffusion model}

\date{March 20, 2023}

\begin{document}
	
	\maketitle
	
	\begin{abstract}
	This paper concerns the numerical solution of the two-dimensional time-dependent partial integro-differential equation 
	(PIDE) that holds for the values of European-style options under the two-asset Kou jump-diffusion model. 
	A main feature of this equation is the presence of a nonlocal double integral term.
	For its numerical evaluation, we extend a highly efficient algorithm derived by Toivanen~\cite{T08} in the case of the 
	one-dimensional Kou integral.
	The acquired algorithm for the two-dimensional Kou integral has optimal computational cost: the number of basic 
	arithmetic operations is directly proportional to the number of spatial grid points in the semidiscretization.
	For the effective discretization in time, we study seven contemporary operator splitting schemes of the 
	implicit-explicit (IMEX) and the alternating direction implicit (ADI) kind.
	All these schemes allow for a convenient, explicit treatment of the integral term. 
	We analyze their (von Neumann) stability.
	By ample numerical experiments for put-on-the-average option values, the actual convergence behavior 
	as well as the mutual performance of the seven operator splitting schemes are investigated. 
	Moreover, the Greeks Delta and Gamma are considered. 
	
	\end{abstract}

	
	\section{Introduction}
	
	In contemporary financial option valuation theory, jump-diffusion processes form a principal class of models for the 
	evolution of the underlying asset prices, see~e.g.~Cont \& Tankov~\cite{CT04book} and Schoutens~\cite{S03book}.
	The first jump-diffusion process was proposed in 1976 by Merton~\cite{M76}.
	In this classical model, the relative jump sizes are assumed to be lognormally distributed.
	A wide variety of jump-diffusion processes, and more generally, exponential L\'{e}vy processes, for asset prices has been introduced 
	in the literature since then, for example the familiar variance gamma (VG), normal inverse Gaussian (NIG) and Carr--Geman--Madan--Yor 
	(CGMY) models~\cite{CT04book,S03book}.
	
	In this paper we consider the jump-diffusion model proposed by Kou~\cite{K02}.
	In this model, the relative jump sizes are given by a log-double-exponential distribution.
	Just like the models mentioned above, Kou's jump-diffusion model has become popular in financial option valuation theory and practice.
	In this paper we are interested in the valuation of European-style options under a direct extension of Kou's single 
	asset jump-diffusion model~\cite{K02} to two assets.
	Here the finite activity jumps in the two asset prices are assumed to occur contemporaneously.
	Financial option valuation theory then yields a two-dimensional time-dependent partial integro-differential equation (PIDE)
	that must be satisfied for the values of European two-asset options. 
	The integral part, which stems from the contribution of the jumps, is nonlocal: it is taken over the whole, two-dimensional asset 
	price domain.
	In general, (semi-)closed analytical solutions to this PIDE are not available in the literature.
	Accordingly, in the present paper, we investigate the effective numerical solution of the two-dimensional time-dependent Kou PIDE.
	
	For the numerical solution, we follow the well-known and general method of lines (MOL) approach.
	The two-dimensional Kou PIDE is first discretized in space by finite differences, and the resulting large, semidiscrete system of 
	ordinary differential equations (ODEs) is subsequently discretized in time by a suitable implicit time stepping scheme.
	The main challenges for the efficient and stable numerical solution are the treatment of the two-dimensional integral part and the 
	treatment of the two-dimensional PDE part, which includes a mixed spatial derivative term.
	
	Spatial discretization of the PIDE leads to a large, dense matrix for the integral part.
	In each time step of the schemes under consideration, products of this matrix with one or more given vectors need to be computed,
	which can form a computational burden.
	For the one-dimensional Kou PIDE, however, Toivanen~\cite{T08} derived a simple algorithm for evaluating these matrix-vector 
	products that has optimal computational cost.
	A key result of our paper is a generalization of Toivanen's algorithm to the two-dimensional Kou PIDE that maintains optimal 
	computational cost.
	
	For the temporal discretization of semidiscretized one-dimensional PIDEs arising in financial option valuation under jump-diffusion 
	processes with finite activity jumps, various authors have proposed operator splitting schemes where the (stiff) PDE part is handled 
	implicitly and the (nonstiff) integral part explicitly.
	Cont \& Voltchkova~\cite{CV05} considered an {\it implicit-explicit (IMEX)} splitting scheme where the PDE part is handled by 
	the backward Euler method and the integral part by the forward Euler method.
	This IMEX Euler scheme is only first-order consistent.
	A variety of higher-order IMEX schemes for PIDEs in finance has been studied since, e.g.~by Briani, Natalini \& Russo~\cite{B07},
	Feng \& Linetsky~\cite{FL08}, Kwon \& Lee~\cite{KL11} and Salmi \& Toivanen~\cite{ST14}.
	The latter authors proposed the IMEX CNAB scheme, where the PDE part is treated by the Crank--Nicolson method and the integral part 
	by the second-order Adams--Bashforth method.
	The IMEX CNAB scheme has been successfully applied to two-dimensional option valuation PIDEs in e.g.~\cite{HT16,STS14}. 
	
	Tavella \& Randall~\cite{TR00book} and d'Halluin, Forsyth \& Vetzal~\cite{HFV05} considered an alternative approach where the PDE 
	part is treated by the Crank--Nicolson method and a fixed-point iteration on the integral part is performed in each time step. 
	This approach has been applied to two-dimensional option valuation PIDEs in Clift \& Forsyth~\cite{CF08}, including the 
	two-dimensional Kou PIDE.
	When the number of fixed-point iterations is frozen, one arrives at a particular IMEX scheme.
	
	For the efficient temporal discretization of semidiscrete two-dimensional PIDEs, a subsequent important improvement is obtained 
	by using, instead of the Crank--Nicolson method, an {\it alternating direction implicit (ADI)} splitting scheme for the 
	two-dimensional PDE part.
	In the computational finance literature, a range of effective, second-order ADI schemes has been developed and analyzed for 
	multi-dimensional PDEs (without integral part), where in each time step the implicit unidirectional stages are combined with 
	explicit stages involving the mixed derivative terms, see e.g.~\cite{H17book,HT16}.
	In this paper, we consider the well-established modified Craig--Sneyd (MCS) scheme, introduced by in 't Hout \& Welfert~\cite{HW09}, 
	and the stabilizing correction two-step Adams-type scheme called SC2A, constructed by Hundsdorfer \& in 't Hout~\cite{HH18}.
	
	The direct adaptation of ADI schemes for PDEs to PIDEs in finance has first been studied by Kaushansky, Lipton \& Reisinger~\cite{KLR18} 
	and next by in 't Hout \& Toivanen~\cite{HT18}.
	Here the implicit unidirectional stages are blended with explicit stages involving both the mixed derivative terms and the integral 
	part, leading again to second-order schemes.
	We note that an efficient, parallel implementation of the schemes developed in \cite{HT18} has been designed by Ghosh \& Mishra~\cite{GM21} 
	who apply a parallel cyclic reduction algorithm.
	
	Boen \& in 't Hout~\cite{BH21} recently investigated a collection of seven contemporary operator splitting schemes of both the 
	IMEX and the ADI kind in the application to the two-dimensional Merton PIDE for the values of two-asset options.
	Here the numerical evaluation of the integral term has been done by means of a FFT-type algorithm, following 
	e.g.~\cite{AO05,AA00,CF08,HFV05,STS14}. 
	Based on analytical and numerical evidence in \cite{BH21,HT18} in the case of the two-dimensional Merton and Bates PIDEs, 
	it is concluded that, among the schemes under consideration, the adaptation of the MCS scheme introduced in \cite{HT18} 
	that deals with the integral part in a two-step Adams--Bashforth fashion is preferable.
	
	In the present paper we consider for the two-dimensional Kou PIDE the same collection of operator splitting schemes as in~\cite{BH21}. 
	For the numerical evaluation of the double integral part, a generalization of the algorithm of Toivanen~\cite{T08} is derived that has 
	optimal computational cost.
	This algorithm is simple to implement, requires little memory usage and is computationally much faster than the FFT-type algorithm 
	mentioned above, cf.~\cite{T08}.
	As a representative example, we consider the approximation of European put-on-the-average option values, together with their Greeks.
	An outline of the rest of this paper is as follows.
	
	In Section~\ref{SecModel} the two-dimensional Kou PIDE is formulated.
	Section~\ref{SecSpatial} deals with its spatial discretization.
	First, in Subsection~\ref{SecPDEpart}, the two-dimensional PDE part is considered and a second-order finite difference discretization 
	on a suitable nonuniform spatial grid is described.  
	Next, in Subsection~\ref{SecIntpart}, we present the first main contribution of this paper.
	A common, second-order spatial discretization of the double integral part is employed and for its highly efficient evaluation 
	we derive an extension of the algorithm proposed by Toivanen~\cite{T08}.
	It is shown that the computational cost of the extension is directly proportional to the number of spatial grid points, 
	which is optimal.
	Section~\ref{SecTime} subsequently concerns the temporal discretization of the obtained semidiscrete two-dimensional Kou PIDE.
	Here the seven contemporary operator splitting schemes of the IMEX and ADI kind from~\cite{BH21} are considered.  
	Each of these schemes conveniently treats the integral part in an explicit manner, where its fast evaluation is performed by the 
	algorithm derived in Subsection~\ref{SecIntpart}.
	A stability analysis of the operator splitting schemes pertinent to two-dimensional PIDEs is given in Section~\ref{SecStab}.
	In Section~\ref{SecResults} ample numerical experiments are presented.
	Here European put-on-the-average option values, together with their Greeks Delta and Gamma, are considered and we examine in detail
	the temporal discretization errors and performance of the different operator splitting schemes.
	The final Section~\ref{SecConc} gives conclusions.

	
	\section{The two-dimensional Kou PIDE}\label{SecModel}
	Under the two-asset Kou jump-diffusion model, the value $v = v(s_1,s_2,t)$ of a European-style option with maturity date $T>0$ and $s_i$ ($i=1,2$) 
	representing the price of asset $i$ at time $\tau = T-t$, satisfies the following PIDE:
	\begin{align}\label{PIDE2D}
		\dud{t} =&~ \tfrac{1}{2} \sigma_1^2s_1^2\dudd{s_1} + \rho \sigma_1\sigma_2s_1s_2\duddm{s_1}{s_2} + \tfrac{1}{2} \sigma_2^2 s_2^2\dudd{s_2} + 
		(r-\lambda \kappa_1) s_1\dud{s_1} + (r-\lambda\kappa_2) s_2 \dud{s_2} \nonumber \\
		& -(r+\lambda)v+\lambda\int_0^{\infty}\int_0^{\infty} f(y_1,y_2) v(s_1y_1, s_2y_2,t) \dd y_1 \dd y_2
	\end{align}
	whenever $s_1>0$, $s_2>0$, $0<t\leq T$.
	Here $r$ is the risk-free interest rate, $\sigma_i > 0$ ($i=1,2$) is the instantaneous volatility for asset $i$ conditional 
	on the event that no jumps occur, and $\rho$ is the correlation coefficient of the two underlying standard Brownian motions.
	Next, $\lambda$ is the jump intensity of the underlying Poisson arrival process, and $\kappa_i$ ($i=1,2$) is the expected 
	relative jump size for asset $i$.
    The function $f$ is the joint probability density function of two independent random variables possessing log-double-exponential 
	distributions~\cite{K02},
	\begin{equation}\label{pdf2D}
		f(y_1,y_2) = 
		\left\{\begin{array}{lll}
			q_1q_2\eta_{q_1}\eta_{q_2}y_1^{\eta_{q_1}-1}y_2^{\eta_{q_2}-1} & (0 < y_1, y_2 < 1),\\
			p_1q_2\eta_{p_1}\eta_{q_2}y_1^{-\eta_{p_1}-1}y_2^{\eta_{q_2}-1} & (y_1 \geq 1, 0< y_2 < 1),\\
			q_1p_2\eta_{q_1}\eta_{p_2}y_1^{\eta_{q_1}-1}y_2^{-\eta_{p_2}-1} & (0 < y_1 < 1, y_2 \geq 1),\\
			p_1p_2\eta_{p_1}\eta_{p_2}y_1^{-\eta_{p_1}-1}y_2^{-\eta_{p_2}-1} & (y_1, y_2 \geq 1).
		\end{array}\right.
	\end{equation}  
	The parameters $p_i$, $q_i$, $\eta_{p_i}$, $\eta_{q_i}$  are all positive constants with $p_i + q_i = 1$ and $\eta_{p_i} > 1$.
	It holds that
	\begin{equation*}
	    \kappa_i = \frac{p_i\eta_{p_i}}{\eta_{p_i}-1}+\frac{q_i\eta_{q_i}}{\eta_{q_i}+1}-1 \quad (i=1,2).
	\end{equation*}
	For \eqref{PIDE2D}, the initial condition is given by
	\[
	v(s_1,s_2,0) = \phi(s_1,s_2),
	\]
	where $\phi$ denotes the payoff function of the option.
	As a typical example, we consider in this paper a European put-on-the-average option, which has the payoff function
	\begin{equation}\label{payoff}
		\phi(s_1,s_2) = \textrm{max} \left(0\,,\,K-\frac{s_1+s_2}{2}\right)
	\end{equation}
	with strike price $K>0$.
	Its graph is shown in Figure~\ref{FigPayoff}.
	Concerning the boundary condition, it holds that the PIDE \eqref{PIDE2D} is itself satisfied on the two sides $s_{1} = 0$ 
	and $s_{2} = 0$, respectively.
	
	\begin{figure}[h]
	\centering
	\includegraphics[trim = 1.3in 3.7in 1in 3.7in, clip, scale=0.5]{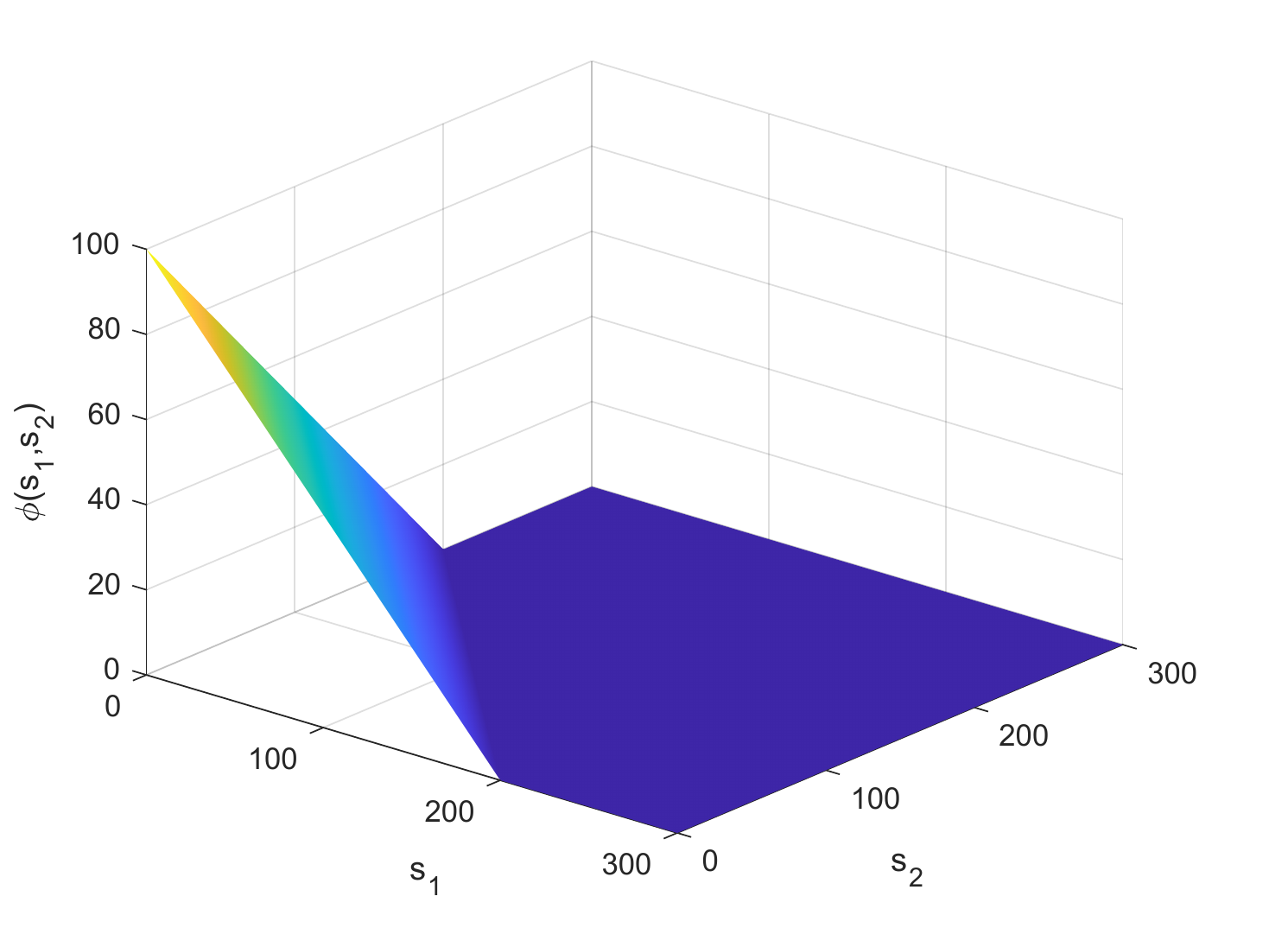}
	\caption{Payoff put-on-the-average option with $K=100$.}
	\label{FigPayoff}
	\end{figure}

	
	\section{Spatial discretization}\label{SecSpatial}
	For the numerical solution of the initial-boundary value problem for \eqref{PIDE2D} we employ the popular method of lines (MOL) approach.
	This approach consists of two consecutive steps~\cite{HV03book}: first the PIDE \eqref{PIDE2D} is discretized in space and subsequently in time.
	This Section~\ref{SecSpatial} deals with the spatial discretization. 
	In the next Section~\ref{SecTime} we shall consider the temporal discretization.
	
	
	\subsection{Convection-diffusion-reaction part}\label{SecPDEpart}
	For the numerical solution, the spatial domain is truncated to a bounded set $[0,S_{\rm max}]\times[0,S_{\rm max}]$ 
	with fixed value $S_{\rm max}$ chosen sufficiently large.
	On the two far sides $s_{1} = S_{\rm max}$ and $s_{2} = S_{\rm max}$ a linear boundary condition is taken, which is well-known in finance,
	\begin{equation}\label{LBC}
		\dudd{s_1} = 0~~(\textrm{if}~s_{1} = S_{\rm max}) \quad \mbox{ and } \quad  \dudd{s_2} = 0 ~~(\textrm{if}~s_{2} = S_{\rm max}).
	\end{equation}
	In this subsection we describe the finite difference discretization of the convection-diffusion-reaction part of the PIDE \eqref{PIDE2D}, 
	specified by
	\begin{equation*}\label{lindiffop}
		\mathcal{D} v =  \tfrac{1}{2} \sigma_1^2s_1^2\dudd{s_1} + \rho \sigma_1\sigma_2s_1s_2\duddm{s_1}{s_2} + \tfrac{1}{2} \sigma_2^2 s_2^2\dudd{s_2} 
		+ (r-\lambda \kappa_1) s_1\dud{s_1} + (r-\lambda\kappa_2) s_2 \dud{s_2} -(r+\lambda)v.
	\end{equation*}
	The semidiscretization of this part is common and similar to that in~e.g.~\cite{BH21}.
	
	Let integers $m_1, m_2 \geq 1$ be given.
	The option value function $v$ is approximated at a nonuniform, Cartesian set of spatial grid points,
	\begin{equation*}\label{sgrid}
		(s_{1,i},s_{2,j})\in [0,S_{\rm max}]\times[0,S_{\rm max}] \quad (0\leq i \leq m_1, \ 0\leq j \leq m_2),
	\end{equation*}
	with $s_{1,0} = s_{2,0} = 0$ and $s_{1,m_1} = s_{2,m_2} = S_{\rm max}$. 
	The nonuniform grid in each spatial direction is defined through a smooth transformation of an artificial uniform grid, such that 
	relatively many grid points are placed in a region of financial and numerical interest.
	Figure~\ref{FigGrid} shows a sample spatial grid if $m_1=m_2=50$, $K=100$ and $S_{\rm max} = 5K$.
	
	Let integer $m \geq 1$ and parameter $d > 0$.
	Consider equidistant points $0 = \xi_0 < \xi_1 < \dots < \xi_{m} = \xi_{\rm max}$ where
	\begin{equation*}
		\xi_{\rm max} = \xi_{\rm int} + \sinh^{-1}\left(\frac{S_{\rm max}}{d}-\xi_{\rm int}\right) 
		\quad \textrm{and } \quad
		\xi_{\rm int} = \frac{2K}{d}.
	\end{equation*}
	Then in each spatial direction a nonuniform mesh $0=s_0<s_1<\cdots<s_m=S_{\rm max}$ is constructed by the smooth transformation 
	$s_{i} = \varphi(\xi_i)$ $(0\le i \le m)$ with
	\begin{equation*}
		\varphi(\xi) = 
		\left\{\begin{array}{lll}
			d \xi &(0 \leq \xi \leq \xi_{\rm int}),\\\\
			2K + d\sinh(\xi-\xi_{\rm int}) &(\xi_{\rm int} < \xi \leq \xi_{\rm max}).
		\end{array}\right.
	\end{equation*}  
	This mesh for $s$ is uniform inside the interval $[0,2K]$ and nonuniform outside.
	The parameter $d$ controls the fraction of points $s_{i}$ that lie inside.
	In this paper we (heuristically) choose $d = K/10$, such that the largest fraction is inside $[0,2K]$.
	
	\begin{figure}[h]
		\centering
		\includegraphics[trim = 1.3in 3.7in 1in 3.7in, clip, scale=0.6]{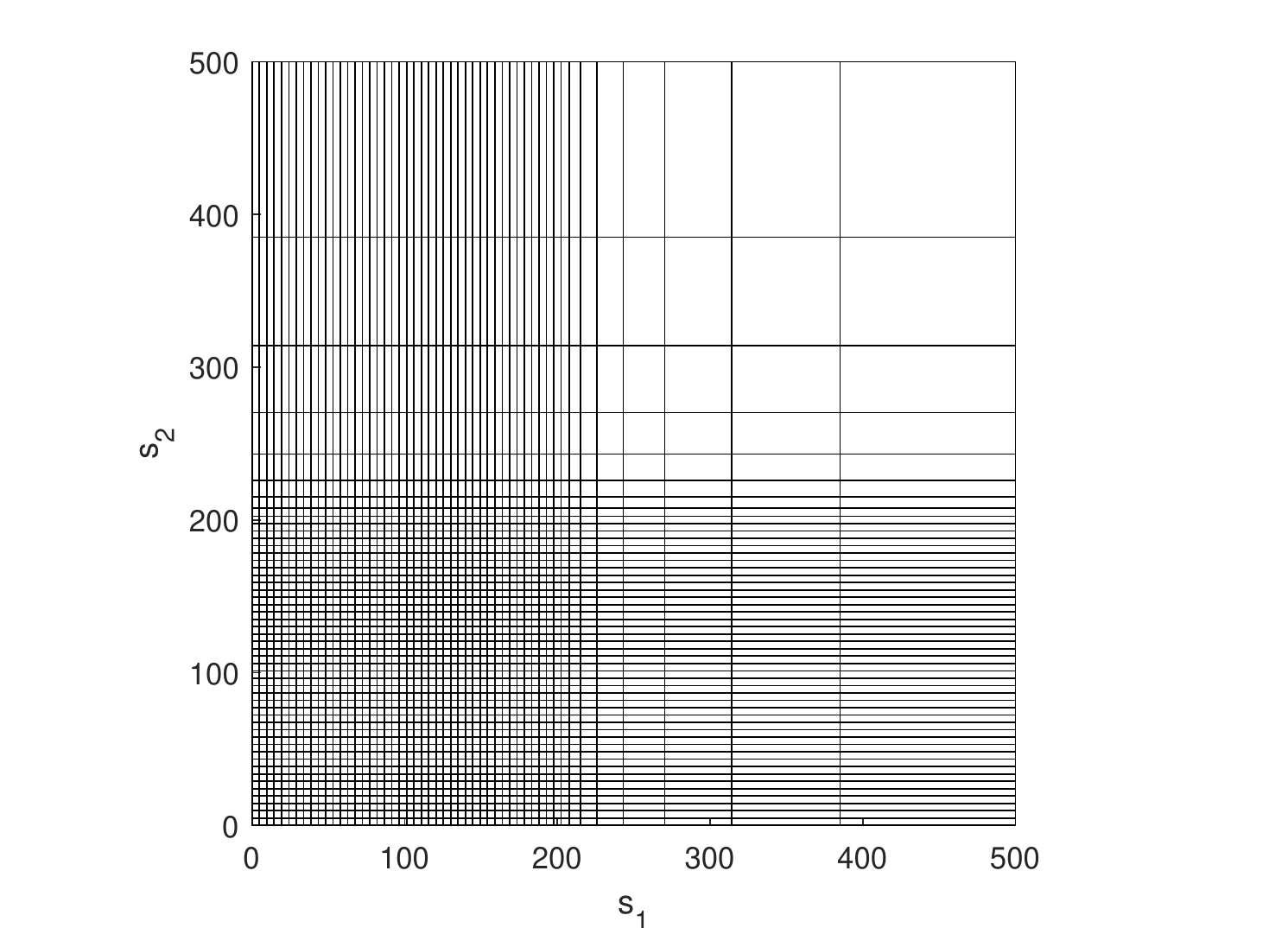}
		\caption{Sample spatial grid for $m_1=m_2=50$, $K=100$, $S_{\rm max} = 5K$.}
		\label{FigGrid}
	\end{figure}
	
	Define mesh widths $h_i = s_i-s_{i-1}$ ($1\leq i \leq m$) and let $u: [0,S_{\rm max}] \to \mathbb{R}$ be any given smooth function. 
	For approximating the first and second derivatives of $u$, the following second-order central finite difference formulas are used:
	\begin{align*}\label{semconv}
		u^\prime (s_i) \approx \omega_{i,-1} u(s_{i-1}) + \omega_{i,0} u(s_{i}) + \omega_{i,1} u(s_{i+1}),
	\end{align*}
	with 
	\[
	\omega_{i,-1} = \frac{-h_{i+1}}{h_{i}(h_{i}+h_{i+1})},\quad
	\omega_{i,0} = \frac{h_{i+1}-h_{i}}{h_{i}h_{i+1}},\quad 
	\omega_{i,1} = \frac{h_{i}}{h_{i+1}(h_{i}+h_{i+1})}
	\]
	and
	\begin{align*}\label{semdiff}
		u^{\prime\prime} (s_i) \approx \omega_{i,-1} u(s_{i-1}) + \omega_{i,0} u(s_{i}) + \omega_{i,1} u(s_{i+1}),
	\end{align*}
	with
	\[
	\omega_{i,-1} = \frac{2}{h_{i}(h_{i}+h_{i+1})},\quad 
	\omega_{i,0} = \frac{-2}{h_{i}h_{i+1}},\quad 
	\omega_{i,1} = \frac{2}{h_{i+1}(h_{i}+h_{i+1})}
	\]
	whenever $1\le i \le m-1$.
	For $i=0$ no finite difference formulas are needed, due to the degeneracy of $\mathcal{D} v$ at the zero boundaries.
	For $i=m$, the first derivative is discretized by the first-order backward finite difference formula and the second 
	derivative is equal to zero by the linear boundary condition \eqref{LBC}.
	Concerning the mixed derivative term in $\mathcal{D} v$, this is approximated by successively applying the relevant 
	finite difference formulas for the first derivative in the two spatial directions.
	
	Let $V_{i,j}(t)$ denote the semidiscrete approximation to $v(s_{1,i},s_{2,j},t)$ for $0\leq i \leq m_1$, $0\leq j \leq m_2$ 
	and define the vector 
	\[
	V(t) = (V_{0,0}(t), V_{1,0}(t),\ldots,V_{m_1-1,m_2}(t),V_{m_1,m_2}(t))^{\top} \in \mathbb{R}^{(m_1+1)(m_2+1)}.
	\] 
	The semidiscretized convection-diffusion-reaction part $\mathcal{D} v$ of PIDE \eqref{PIDE2D} is then given by
	\[
	A^{(D)} V(t)
	\] 
	with matrix
	\[
	A^{(D)} = A^{(M)} + A_1 + A_2,
	\]
	where
	\begin{align*}
		A^{(M)} &= \rho\sigma_1\sigma_2\left(X_2D_2^{(1)}\right)\otimes\left(X_1D_1^{(1)}\right),\\
		A_1 &= I_2 \otimes \left(\tfrac{1}{2}\sigma_1^2X_1^2D_1^{(2)} + (r-\lambda\kappa_1)X_1D_1^{(1)}-\tfrac{1}{2}(r+\lambda)I_1\right),\\
		A_2 &= \left(\tfrac{1}{2}\sigma_2^2X_2^2D_2^{(2)} + (r-\lambda\kappa_2)X_2D_2^{(1)}-\tfrac{1}{2}(r+\lambda)I_2\right)\otimes I_1.
	\end{align*}
	Here $I_k$, $X_k$, $D_k^{(l)}$ are given $(m_k+1)\times (m_k+1)$ matrices for $k,l\in \{1,2\}$ with $I_k$ being the identity matrix, 
	$X_k$ being the diagonal matrix
	\begin{equation*}
		X_k = {\rm diag}(s_{k,0}, s_{k,1},\ldots, s_{k,m_k})
	\end{equation*}
	and $D_k^{(l)}$ the matrix representing numerical differentiation of order $l$ in the $k$-th spatial direction by the relevant 
	finite difference formula above.
	The matrix $A^{(M)}$ corresponds to the mixed derivative term in $\mathcal{D} v$ and $A_k$ corresponds to all derivative terms 
	in the $k$-th spatial direction ($k=1,2$), where the reaction term has been distributed evenly across $A_1$ and $A_2$.
	These matrices are all sparse. In particular, $A_1$ and $A_2$ are (essentially) tridiagonal.

	
\subsection{Integral part}\label{SecIntpart}
In the following we consider discretization of the double integral
\begin{equation}\label{integral}
	\mathcal{J} = \lambda\int_0^{\infty}\int_0^{\infty} f(y_1,y_2)v(s_1y_1, s_2y_2,t) \dd y_1 \dd y_2
\end{equation}
on the spatial grid from Subsection~\ref{SecPDEpart}.
For its efficient evaluation, we derive an extension of the algorithm proposed by Toivanen~\cite{T08} for the special case of the 
one-dimensional Kou model.

Assume $s_1, s_2>0$ are given. A change of variables $y_i = z_i/s_i$ ($i=1,2$) yields
\begin{equation*}
	\mathcal{J} = \lambda\int_0^{\infty}\int_0^{\infty} f\bigg(\frac{z_1}{s_1},\frac{z_2}{s_2}\bigg) v(z_1,z_2,t) \frac{\dd z_1 \dd z_2}{s_1 s_2}.
\end{equation*}
The density function $f$ is defined on a partition of four sets of the first quadrant in the real plane, see (\ref{pdf2D}).
It follows that $\mathcal{J}$ is decomposed into four integrals as $\mathcal{J} = \mathcal{J}_1 + \mathcal{J}_2 + \mathcal{J}_3 + \mathcal{J}_4$, where
\begin{align*}
	\mathcal{J}_1 &= \lambda q_1q_2\eta_{q_1}\eta_{q_2}s_1^{-\eta_{q_1}}s_2^{-\eta_{q_2}} \int_0^{s_2} \int_0^{s_1} z_1^{\eta_{q_1}-1}z_2^{\eta_{q_2}-1} v(z_1,z_2,t) \dd z_1 \dd z_2,\\
	\mathcal{J}_2 &= \lambda p_1q_2\eta_{p_1}\eta_{q_2}s_1^{\eta_{p_1}}s_2^{-\eta_{q_2}}  \int_0^{s_2} \int_{s_1}^{\infty} z_1^{-\eta_{p_1}-1}z_2^{\eta_{q_2}-1} v(z_1,z_2,t) \dd z_1 \dd z_2,\\
	\mathcal{J}_3 &= \lambda q_1p_2\eta_{q_1}\eta_{p_2}s_1^{-\eta_{q_1}}s_2^{\eta_{p_2}}  \int_{s_2}^{\infty} \int_0^{s_1} z_1^{\eta_{q_1}-1}z_2^{-\eta_{p_2}-1} v(z_1,z_2,t) \dd z_1 \dd z_2,\\		
	\mathcal{J}_4 &= \lambda p_1p_2\eta_{p_1}\eta_{p_2}s_1^{\eta_{p_1}}s_2^{\eta_{p_2}}   \int_{s_2}^{\infty} \int_{s_1}^{\infty} z_1^{-\eta_{p_1}-1}z_2^{-\eta_{p_2}-1} v(z_1,z_2,t) \dd z_1 \dd z_2.
\end{align*}
We first consider discretization of the integral $\mathcal{J}_1$. 
Upon writing
\begin{equation*}
	\psi_1(s_1,s_2) = \lambda q_1q_2\eta_{q_1}\eta_{q_2}s_1^{-\eta_{q_1}}s_2^{-\eta_{q_2}} 
	\quad \textrm{and} \quad
	\varphi_1(z_1,z_2) = z_1^{\eta_{q_1}-1}z_2^{\eta_{q_2}-1},
\end{equation*}
we have
\begin{equation}
	\mathcal{J}_1 = \psi_1(s_1,s_2)\int_0^{s_2} \int_0^{s_1} \varphi_1(z_1,z_2) v(z_1,z_2,t) \dd z_1 \dd z_2.
\end{equation}
For $1\leq i \leq m_1$, $1\leq j \leq m_2$ let
\begin{equation*}
	\mathcal{J}_{1,ij} = \psi_1(s_{1,i},s_{2,j}) \int_0^{s_{2,j}} \int_0^{s_{1,i}} \varphi_1(z_1,z_2) v(z_1,z_2,t) \dd z_1 \dd z_2
\end{equation*}
denote the value of $\mathcal{J}_1$ at the spatial grid point $(s_{1,i},s_{2,j})$.
Define
\begin{equation*}
	\mathcal{G}_{1,kl} = \int_{s_{2,l-1}}^{s_{2,l}} \int_{s_{1,k-1}}^{s_{1,k}} \varphi_1(z_1,z_2) v(z_1,z_2,t) \dd z_1 \dd z_2
\end{equation*}
whenever $1\leq k \leq m_1$, $1\leq l \leq m_2$.
Then the following useful expression for $\mathcal{J}_{1,ij}$ in terms of a double cumulative sum is obtained,
\begin{equation}\label{cumsum1}
	\mathcal{J}_{1,ij} = \psi_1(s_{1,i},s_{2,j}) \sum_{k=1}^i \sum_{l=1}^j \mathcal{G}_{1,kl}
	\quad
	(1\leq i \leq m_1, 1\leq j \leq m_2).
\end{equation}
Notice the obvious but important fact that the $\mathcal{G}_{1,kl}$ are independent of the indices $i$ and $j$.
Hence, if all values $\mathcal{G}_{1,kl}$ are given, then computing the double cumulative sums in \eqref{cumsum1} 
for all $i$, $j$ can be done in, to leading order, just $2m_1m_2$ additions.

We subsequently construct approximations $G_{1,kl}$ to $\mathcal{G}_{1,kl}$ ($1\leq k \leq m_1$, $1\leq l \leq m_2$)
and define the approximation to $\mathcal{J}_{1,ij}$ by
\begin{equation}\label{cumsum1_semi}
	J_{1,ij} = \psi_1(s_{1,i},s_{2,j}) \sum_{k=1}^i \sum_{l=1}^j G_{1,kl}
	\quad
	(1\leq i \leq m_1, 1\leq j \leq m_2).
\end{equation}
For any given $k, l$ with $1\leq k \leq m_1$, $1\leq l \leq m_2$ consider the natural choice of bilinear interpolation 
to approximate $v(z_1,z_2,t)$ on the $(z_1,z_2)$-domain $[s_{1,k-1}, s_{1,k}] \times [s_{2,l-1}, s_{2,l}]$:
\begin{equation*}
	{\widetilde v}_{kl}(z_1, z_2,t) = \ell_{kl}^{00}(z_1,z_2) V_{k-1,l-1}(t) + \ell_{kl}^{10}(z_1,z_2) V_{k,l-1}(t) +
	\ell_{kl}^{01}(z_1,z_2) V_{k-1,l}(t) + \ell_{kl}^{11}(z_1,z_2) V_{k,l}(t)
\end{equation*}
with weights
\begin{align*}
	\ell_{kl}^{00}(z_1,z_2) &= (s_{1,k}-z_1)(s_{2,l}-z_2)/\delta_{kl},\\
	\ell_{kl}^{10}(z_1,z_2) &= (z_1-s_{1,k-1})(s_{2,l}-z_2)/\delta_{kl},\\
	\ell_{kl}^{01}(z_1,z_2) &= (s_{1,k}-z_1)(z_{2}-s_{2,l-1})/\delta_{kl},\\
	\ell_{kl}^{11}(z_1,z_2) &= (z_1-s_{1,k-1})(z_2-s_{2,l-1})/\delta_{kl},
\end{align*}
where $\delta_{kl} = \Delta s_{1,k}\Delta s_{2,l}$ and $\Delta s_{1,k} = s_{1,k}-s_{1,k-1}$, $\Delta s_{2,l} = s_{2,l}-s_{2,l-1}$.
Then we define 
\begin{equation*}
	G_{1,kl} = \int_{s_{2,l-1}}^{s_{2,l}} \int_{s_{1,k-1}}^{s_{1,k}} \varphi_1(z_1,z_2) {\widetilde v}_{kl}(z_1, z_2,t) \dd z_1 \dd z_2.
\end{equation*}
A straightforward calculation yields the simple, convenient formula 
\begin{equation}\label{G1kl}
	G_{1,kl} = 
	\gamma_{1,kl}^{00} V_{k-1,l-1}(t) + \gamma_{1,kl}^{10} V_{k,l-1}(t) +
	\gamma_{1,kl}^{01} V_{k-1,l}(t)   + \gamma_{1,kl}^{11} V_{k,l}(t)
\end{equation}
with
\begin{align*}
	\gamma_{1,kl}^{00} &= (s_{1,k}s_{2,l}\zeta_{1,kl}^{00} - s_{2,l}\zeta_{1,kl}^{10} - s_{1,k}\zeta_{1,kl}^{01} + \zeta_{1,kl}^{11})/\delta_{kl},\\
	\gamma_{1,kl}^{10} &= (-s_{1,k-1}s_{2,l}\zeta_{1,kl}^{00} + s_{2,l}\zeta_{1,kl}^{10} + s_{1,k-1}\zeta_{1,kl}^{01} - \zeta_{1,kl}^{11} )/\delta_{kl},\\
	\gamma_{1,kl}^{01} &= (-s_{1,k}s_{2,l-1}\zeta_{1,kl}^{00} + s_{2,l-1}\zeta_{1,kl}^{10} + s_{1,k}\zeta_{1,kl}^{01} - \zeta_{1,kl}^{11} )/\delta_{kl},\\
	\gamma_{1,kl}^{11} &= (s_{1,k-1}s_{2,l-1}\zeta_{1,kl}^{00} - s_{2,l-1}\zeta_{1,kl}^{10} - s_{1,k-1}\zeta_{1,kl}^{01} + \zeta_{1,kl}^{11})/\delta_{kl}
\end{align*}
and
\begin{equation*}
	\zeta_{1,kl}^{ab} =  
	\int_{s_{2,l-1}}^{s_{2,l}} \int_{s_{1,k-1}}^{s_{1,k}} \varphi_1(z_1,z_2) z_1^a z_2^b \dd z_1 \dd z_2 =
	\frac{\bigg(s_{1,k}^{a+\eta_{q_1}}-s_{1,k-1}^{a+\eta_{q_1}}\bigg)\bigg(s_{2,l}^{b+\eta_{q_2}}-s_{2,l-1}^{b+\eta_{q_2}}\bigg)}{\big(a+\eta_{q_1}\big)\big(b+\eta_{q_2}\big)}
\end{equation*}
for $a,b\in \{0,1\}$.

The coefficients $\gamma_{1,kl}^{ab}$ are completely determined by the Kou parameters and the spatial grid. 
Since they are independent of $t$, they can be computed upfront, before the time discretization.
Clearly, for any given vector $V(t)$, the computation of $G_{1,kl}$ by \eqref{G1kl} for all $k$, $l$ requires 
$3m_1m_2$ additions and $4m_1m_2$ multiplications.
Noticing that the values of $\psi_1$ in \eqref{cumsum1_semi} can also be computed upfront, it follows that 
the number of basic arithmetic operations to compute all approximations $J_{1,ij}$ ($1\leq i \leq m_1$, 
$1\leq j \leq m_2$) by \eqref{cumsum1_semi} is, to leading order, equal to $10m_1m_2$.

The discretization and efficient evaluation of the other three integrals is done completely analogously.
In the relevant derivation, $v$ is approximated by zero outside the spatial domain $[0,S_{\rm max}]\times [0,S_{\rm max}]$.
Write
\begin{align*}
	\psi_2(s_1,s_2) &= \lambda p_1q_2\eta_{p_1}\eta_{q_2}s_1^{\eta_{p_1}}s_2^{-\eta_{q_2}},\\
	\psi_3(s_1,s_2) &= \lambda q_1p_2\eta_{q_1}\eta_{p_2}s_1^{-\eta_{q_1}}s_2^{\eta_{p_2}},\\ 
	\psi_4(s_1,s_2) &= \lambda p_1p_2\eta_{p_1}\eta_{p_2}s_1^{\eta_{p_1}}s_2^{\eta_{p_2}}.
\end{align*}
Let $1\leq i \leq m_1$, $1\leq j \leq m_2$. 
Then the approximations of $\mathcal{J}_2, \mathcal{J}_3, \mathcal{J}_4$ 
at the spatial grid point $(s_{1,i},s_{2,j})$ are given by, respectively, the double cumulative sums
\begin{align*}
	J_{2,ij} &= \psi_2(s_{1,i},s_{2,j}) \sum_{k=i+1}^{m_1} \sum_{l=1}^j G_{2,kl},\\
	J_{3,ij} &= \psi_3(s_{1,i},s_{2,j}) \sum_{k=1}^{i}     \sum_{l=j+1}^{m_2} G_{3,kl},\\
	J_{4,ij} &= \psi_4(s_{1,i},s_{2,j}) \sum_{k=i+1}^{m_1} \sum_{l=j+1}^{m_2} G_{4,kl}
\end{align*}
with the usual convention that empty sums are equal to zero.
For $2\le \nu \le 4$, $1\leq k \leq m_1$, $1\leq l \leq m_2$ we obtain
\begin{equation*}\label{Gnukl}
	G_{\nu,kl} = 
	\gamma_{\nu,kl}^{00} V_{k-1,l-1}(t) + \gamma_{\nu,kl}^{10} V_{k,l-1}(t) +
	\gamma_{\nu,kl}^{01} V_{k-1,l}(t)   + \gamma_{\nu,kl}^{11} V_{k,l}(t)
\end{equation*}
with
\begin{align*}
	\gamma_{\nu,kl}^{00} &= (s_{1,k}s_{2,l}\zeta_{\nu,kl}^{00} - s_{2,l}\zeta_{\nu,kl}^{10} - s_{1,k}\zeta_{\nu,kl}^{01} + \zeta_{\nu,kl}^{11})/\delta_{kl},\\
	\gamma_{\nu,kl}^{10} &= (-s_{1,k-1}s_{2,l}\zeta_{\nu,kl}^{00} + s_{2,l}\zeta_{\nu,kl}^{10} + s_{1,k-1}\zeta_{\nu,kl}^{01} - \zeta_{\nu,kl}^{11} )/\delta_{kl},\\
	\gamma_{\nu,kl}^{01} &= (-s_{1,k}s_{2,l-1}\zeta_{\nu,kl}^{00} + s_{2,l-1}\zeta_{\nu,kl}^{10} + s_{1,k}\zeta_{\nu,kl}^{01} - \zeta_{\nu,kl}^{11} )/\delta_{kl},\\
	\gamma_{\nu,kl}^{11} &= (s_{1,k-1}s_{2,l-1}\zeta_{\nu,kl}^{00} - s_{2,l-1}\zeta_{\nu,kl}^{10} - s_{1,k-1}\zeta_{\nu,kl}^{01} + \zeta_{\nu,kl}^{11})/\delta_{kl}
\end{align*}
where
\begin{align*}
	\zeta_{2,kl}^{ab} &=  
	\frac{\bigg(s_{1,k}^{a-\eta_{p_1}}-s_{1,k-1}^{a-\eta_{p_1}}\bigg)\bigg(s_{2,l}^{b+\eta_{q_2}}-s_{2,l-1}^{b+\eta_{q_2}}\bigg)}{\big(a-\eta_{p_1}\big)\big(b+\eta_{q_2}\big)},\\
	\zeta_{3,kl}^{ab} &=  
	\frac{\bigg(s_{1,k}^{a+\eta_{q_1}}-s_{1,k-1}^{a+\eta_{q_1}}\bigg)\bigg(s_{2,l}^{b-\eta_{p_2}}-s_{2,l-1}^{b-\eta_{p_2}}\bigg)}{\big(a+\eta_{q_1}\big)\big(b-\eta_{p_2}\big)},\\
	\zeta_{4,kl}^{ab} &=  
	\frac{\bigg(s_{1,k}^{a-\eta_{p_1}}-s_{1,k-1}^{a-\eta_{p_1}}\bigg)\bigg(s_{2,l}^{b-\eta_{p_2}}-s_{2,l-1}^{b-\eta_{p_2}}\bigg)}{\big(a-\eta_{p_1}\big)\big(b-\eta_{p_2}\big)}
\end{align*}
for $a,b\in \{0,1\}$.

On the boundary part $\{(s_1,0): s_1>0\}$ the double integral \eqref{integral} reduces to the single integral
\begin{equation}\label{integral1_1D}
	\mathcal{J} = \lambda\int_0^{\infty} f_1(y_1)v(s_1y_1,0,t) \dd y_1.
\end{equation}
The discretization and efficient evaluation of this integral, which arises in the one-dimensional 
Kou PIDE, has been constructed by Toivanen~\cite{T08}.
For completeness, we include the result here.
Assume $s_1>0$. 
Then for \eqref{integral1_1D} there holds $\mathcal{J} = \mathcal{J}_1 + \mathcal{J}_2$, where
\begin{equation*}
	\mathcal{J}_1 = \lambda \int_0^{s_1} f_1\bigg(\frac{z_1}{s_1}\bigg) v(z_1,0,t) \frac{\dd z_1 }{s_1}
	\quad \textrm{and} \quad
	\mathcal{J}_2 = \lambda \int_{s_1}^{\infty} f_1\bigg(\frac{z_1}{s_1}\bigg) v(z_1,0,t) \frac{\dd z_1 }{s_1}.
\end{equation*}
Write
\begin{equation*}
	\psi_1(s_1) = \lambda q_1\eta_{q_1}s_1^{-\eta_{q_1}} 
	\quad \textrm{and} \quad
	\psi_2(s_1) = \lambda p_1\eta_{p_1}s_1^{\eta_{p_1}}.
\end{equation*}
Let $1\leq i \leq m_1$. 
Then the approximations of $\mathcal{J}_1, \mathcal{J}_2$ at the spatial grid point $(s_{1,i},0)$ are given by, 
respectively, the single cumulative sums
\begin{equation*}
	J_{1,i} = \psi_1(s_{1,i}) \sum_{k=1}^{i} G_{1,k}
	\quad \textrm{and} \quad
	J_{2,i} = \psi_2(s_{1,i}) \sum_{k=i+1}^{m_1} G_{2,k}.
\end{equation*}
For $1\le \nu \le 2$, $1\leq k \leq m_1$ we obtain, by employing linear interpolation,
\begin{equation*}\label{Gnuk}
	G_{\nu,k} = 
	\gamma_{\nu,k}^{0} V_{k-1,0}(t) + \gamma_{\nu,k}^{1} V_{k,0}(t)
\end{equation*}
with
\begin{equation*}
	\gamma_{\nu,k}^{0} = (s_{1,k}\zeta_{\nu,k}^{0} - \zeta_{\nu,k}^{1})/\Delta s_{1,k}
	\quad \textrm{and} \quad
	\gamma_{\nu,k}^{1} = (-s_{1,k-1}\zeta_{\nu,k}^{0} + \zeta_{\nu,k}^{1} )/\Delta s_{1,k}
\end{equation*}
where
\begin{equation*}
	\zeta_{1,k}^{a} =  
	\frac{s_{1,k}^{a+\eta_{q_1}}-s_{1,k-1}^{a+\eta_{q_1}}}{a+\eta_{q_1}}
	\quad \textrm{and} \quad
	\zeta_{2,k}^{a} =  
	\frac{s_{1,k}^{a-\eta_{p_1}}-s_{1,k-1}^{a-\eta_{p_1}}}{a-\eta_{p_1}}
	\quad (a=0,1).
\end{equation*}

On the boundary part $\{(0,s_2): s_2>0\}$ the double integral \eqref{integral} reduces 
to the single integral
\begin{equation}\label{integral2_1D}
	\mathcal{J} = \lambda\int_0^{\infty} f_2(y_2)v(0,s_2y_2,t) \dd y_2
\end{equation}
and the discretization is performed completely similarly as above. 
Finally, at the spatial grid point $(s_1,s_2)=(0,0)$ it holds that 
$\mathcal{J} \approx  \lambda V_{0,0}(t)$.

From the above it follows that the number of basic arithmetic operations to compute, for any given $t$, the approximation to the double 
integral \eqref{integral} on the full spatial grid is, to leading order, equal to $40m_1m_2$.
 
As an illustration, Table~\ref{cpu_integral} shows the obtained CPU times (in seconds) for our implementation\footnote{In Matlab version 
R2020b, on an Intel Core i7-8665U processor at 1.9 GHz with 16 GB memory.} of the acquired algorithm for approximating 
\eqref{integral} for a single time $t$ on the full spatial grid with $m_1=m_2=m$ and $m=100, 200, \ldots, 1000$ in the case of parameter 
set~1 from Section~\ref{SecResults}.
It is readily verified from Table~\ref{cpu_integral} that the CPU time behaves, for large $m$, directly proportional to $m^2$, which 
agrees with the above theoretical result. 

\begin{table}[H]
    \centering
	\caption{CPU time vs.~$m_1=m_2=m$ for the algorithm of Subsection~\ref{SecIntpart} to approximate the integral \eqref{integral}.}
	\vspace{0.3cm}
		\begin{tabular}{@{}lcccccccccc@{}}
		\toprule
		$m$ & 100 & 200 & 300 & 400 & 500 & 600 & 700 & 800 & 900 & 1000 \\ 
		\midrule
		CPU time (s) & 1.2e-3 & 2.9e-3 & 8.8e-3 & 2.3e-2 & 3.7e-2 & 5.2e-2 & 7.3e-2 & 9.7e-2 & 1.3e-1 & 1.6e-1 \\ 
		\bottomrule
	\end{tabular}
	\label{cpu_integral}
\end{table}

	
\subsection{Semidiscrete PIDE}\label{SecSemiPIDE}
Combining the above semidiscretizations for the convection-diffusion-reaction and integral parts of PIDE \eqref{PIDE2D}, the following 
large system of ODEs is obtained:
\begin{align}\label{ODEs}
	V'(t) = AV(t) \quad (0<t\leq T),
\end{align}
where 
\begin{equation*}
	A = A^{(D)} + A^{(J)} = A^{(M)} + A_1 + A_2 + A^{(J)}.
\end{equation*}
Even though we never actually compute it in this way, the approximation to the double integral \eqref{integral} 
defined in Subsection~\ref{SecIntpart} is formally denoted here by $A^{(J)} V(t)$ with given (full) matrix $A^{(J)}$.

The initial vector $V(0)=V^0$ is defined by pointwise evaluation of the payoff function $\phi$ on the spatial grid, 
except at those grid points that lie near to the line segment given by $s_1+s_2 = 2K$ where $\phi$ is nonsmooth.
It is well-known that using pointwise values of $\phi$ in such a region can lead to a deteriorated (spatial) 
convergence behavior, which can be alleviated by applying cell averaging. 
Accordingly, for $k=1,2$ let
\begin{eqnarray*}
	&s_{k,l+1/2} = \tfrac{1}{2}(s_{k,l}+s_{k,l+1})~~~~ &{\rm if}~~ 0\le l < m,\\
	&h_{k,l+1/2} = s_{k,l+1/2}-s_{k,l-1/2} &{\rm if}~~ 0\le l \le m,
\end{eqnarray*}
with $s_{k,-1/2} = -s_{k,1/2}$ and $s_{k,m+1/2} = S_{\rm max}$.
Then we define \cite{H17book}
\begin{equation}\label{cellave}
	V_{i,j}(0) = \frac{1}{h_{1,i+1/2}h_{2,j+1/2}}\int_{s_{1,i-1/2}}^{s_{1,i+1/2}}\int_{s_{2,j-1/2}}^{s_{2,j+1/2}} \phi(s_1,s_2) \dd s_2 \dd s_1
\end{equation}
whenever the cell 
\[
[s_{1,i-1/2}, s_{1,i+1/2})\times[s_{2,j-1/2}, s_{2,j+1/2})
\]
has a nonempty intersection with the line segment $s_1+s_2 = 2K$. 
Note that the double integral \eqref{cellave} is easily calculated.


\section{Temporal discretization}\label{SecTime}
For the temporal discretization of the semidiscrete PIDE \eqref{ODEs} we study seven operator splitting schemes of the IMEX and ADI 
kind~\cite{BH21}.
Let integer $N\ge 1$ be given and define step size $\Delta t = T/N$ and temporal grid points $t_n = n \Delta t$ for $n=0,1,2,\ldots,N$.
Each of the following schemes generates an approximation $V^{n}$ to $V(t_n)$ successively for $n=1,2,\ldots,N$.
Here the integral part of \eqref{ODEs}, corresponding to the matrix $A^{(J)}$, is always conveniently treated in an explicit manner
and evaluated by applying the efficient algorithm derived in Subsection~\ref{SecIntpart}.

\begin{enumerate}
	
	\item \textit{Crank--Nicolson Forward Euler (CNFE) scheme}:
	\vskip0.1cm
	In this one-step IMEX method, the convection-diffusion-reaction part is handled implicitly by the Crank--Nicolson scheme and the 
	integral part is treated in a simple, explicit Euler fashion:
	\begin{align}\label{CNFE}
		\left(I-\tfrac{1}{2}\Delta t A^{(D)}\right)V^{n} = \left(I + \tfrac{1}{2}\Delta t A^{(D)}\right)V^{n-1} + \Delta t A^{(J)} V^{n-1}.
	\end{align}
	Due to the application of explicit Euler, the order\footnote{This means the classical order of consistency, i.e., for fixed 
		nonstiff ODEs.} of the CNFE scheme is just equal to one. 
	\vskip0.2cm
	\newpage
	
	\item \textit{Crank--Nicolson scheme with fixed-point iteration (CNFI)}:
	\vskip0.1cm
	Tavella \& Randall \cite{TR00book} and d'Halluin, Forsyth \& Vetzal \cite{HFV05} proposed to combine the Crank--Nicolson scheme for 
	the convection-diffusion-reaction part with a fixed-point iteration on the integral part:
	\begin{align}\label{CNFI}
		\left(I-\tfrac{1}{2}\Delta t A^{(D)}\right)Y_k = \left(I + \tfrac{1}{2}\Delta t A^{(D)}\right)V^{n-1} +  \tfrac{1}{2} \Delta t A^{(J)} (Y_{k-1}+V^{n-1})
	\end{align}
	for $k=1,2,\ldots,l$ and $V^n = Y_l$ with starting vector $Y_0 = V^{n-1}$.
	Numerical evidence in \cite{CF08,HFV05} indicates that in general $l=2$ or $l=3$ iterations suffice.
	In the numerical experiments in Section~\ref{SecResults} we shall consider a fixed number of $l=2$ iterations.
	Then method \eqref{CNFI} can be viewed as a one-step IMEX scheme, where the integral part is treated by the explicit trapezoidal 
	rule, also called the modified Euler method. 
	By Taylor expansion, it can be seen that its order is equal to two.
	
	\item \textit{Implicit-Explicit Trapezoidal Rule (IETR)}:
	\vskip0.1cm
	Another blend of the implicit trapezoidal rule (Crank--Nicolson) for the convection-diffusion-reaction part and the explicit 
	trapezoidal rule for the integral part has been considered in~\cite{H17book}:
	\begin{equation}\label{IETR}
		\left\{\begin{array}{lll}
			Y_0 = V^{n-1} + \Delta t\, (A^{(D)}+A^{(J)})V^{n-1},\\
			\widehat{Y}_0 = Y_0+\tfrac{1}{2} \Delta t A^{(J)} \big(Y_0-V^{n-1}\big),\\
			Y_1 = \widehat{Y}_0+\tfrac{1}{2} \Delta t A^{(D)} \big(Y_1-V^{n-1}\big),\\
			V^n = Y_1.
		\end{array}\right.
	\end{equation} 
	Method~\eqref{IETR} also forms a one-step IMEX scheme with order equal to two.  
	It is already well-known in the literature on the numerical solution of PDEs (without integral part), see e.g.~Hundsdorfer \& 
	Verwer~\cite{HV03book}.
	Notice that the first two internal stages $Y_0$, $\widehat{Y}_0$ are explicit, whereas the third internal stage $Y_1$ is implicit.
	\vskip0.2cm 
	
	\item \textit{Crank--Nicolson Adams--Bashforth (CNAB) scheme}:
	\vskip0.1cm
	The CNAB method has been amply considered in the literature on the numerical solution of PDEs as well, 
	see e.g.~\cite{FHV97,HV03book,H02}.
	It has been investigated for the numerical solution of PIDEs by Salmi \& Toivanen \cite{ST14} and Salmi, Toivanen \& von Sydow 
	\cite{STS14}. 
	The convection-diffusion-reaction part is again handled by the Crank--Nicolson scheme, but the integral part is now dealt with 
	in a two-step Adams--Bashforth fashion:
	\begin{align}\label{CNAB}
		\left(I-\tfrac{1}{2} \Delta t A^{(D)}\right)V^{n} = \left(I+\tfrac{1}{2} \Delta t A^{(D)}\right)V^{n-1} + \tfrac{1}{2}\Delta t A^{(J)}(3V^{n-1}-V^{n-2}).
	\end{align}
	Method \eqref{CNAB} constitutes a two-step IMEX scheme and is of order equal to two. 
	It is interesting to remark that \eqref{CNAB} can also be viewed as obtained from \eqref{CNFI}, upon taking $l=1$ and 
	improved (linearly extrapolated) starting value $Y_0 = 2V^{n-1}-V^{n-2}$.
	\vskip0.2cm 
	
\end{enumerate}

As for the spatial discretization, it is well-known that the nonsmoothness of the payoff function can have an unfavorable 
impact on the convergence behavior of the temporal discretization, even if cell averaging is applied, cf.~\cite{H17book}.
For the Crank--Nicolson scheme a common approach to alleviate this has been proposed by Rannacher~\cite{R84}. 
It consists of replacing (only) the first time step by two time steps with step size $\Delta t/2$ using the implicit Euler method, 
thus defining the approximation $V^1$ to $V(t_1)$.
In the present case of two-dimensional PIDEs, the implicit Euler method can be computationally intensive.
We therefore consider, for all four IMEX schemes above, the {\it IMEX Euler scheme} as the starting method:
\begin{align*}    
	\left(I-\tfrac{1}{2}\Delta t A^{(D)}\right)V^{\tfrac{1}{2}} &= V^{0} + \tfrac{1}{2}\Delta t A^{(J)} V^{0},\\
	\left(I-\tfrac{1}{2}\Delta t A^{(D)}\right)V^{1} &= V^{\tfrac{1}{2} } + \tfrac{1}{2}\Delta t A^{(J)} V^{\tfrac{1}{2}}.
\end{align*} 

Clearly, in each of the IMEX schemes under consideration, linear systems need to be solved involving the matrix 
$I-\tfrac{1}{2}\Delta t A^{(D)}$. 
Whereas this matrix is sparse, with at most nine nonzero entries per row, it possesses a large bandwidth that is directly 
proportional to $m_1$.
With a view to obtaining a further reduction in computational work, we consider three recent operator splitting 
schemes for PIDEs that employ also the (directional) splitting $A^{(D)} = A^{(M)} + A_1 + A_2$ of the two-dimensional 
convection-diffusion-reaction part, such that only linear systems with tridiagonal matrices arise.

\begin{enumerate} 
	
	\item[5.] \textit{One-step adaptation of the modified Craig--Sneyd (MCS) scheme}:
	\vskip0.1cm
	The MCS scheme is an ADI scheme that was introduced by in 't Hout \& Welfert \cite{HW09} for the numerical solution of 
	PDEs containing mixed spatial derivative terms.
	The following, direct adaptation to PIDEs has recently been investigated by in 't Hout \& Toivanen \cite{HT18}:
	\begin{equation}\label{MCS}
		\left\{\begin{array}{lll}
			Y_0 = V^{n-1} + \Delta t\, (A^{(D)}+A^{(J)})V^{n-1},\\
			Y_j = Y_{j-1} + \theta\Delta t A_j(Y_j-V^{n-1}) \quad (j=1,2),\\
			\widehat{Y}_0 = Y_0 + \theta \Delta t\, (A^{(M)}+A^{(J)})(Y_2 - V^{n-1}),\\
			\widetilde{Y}_0 = \widehat{Y}_0 + (\tfrac{1}{2}-\theta)\Delta t\, (A^{(D)}+A^{(J)})(Y_2-V^{n-1}),\\
			\widetilde{Y}_j = \widetilde{Y}_{j-1} + \theta \Delta t A_j(\widetilde{Y}_j - V^{n-1}) \quad (j=1,2),\\
			V^n = \widetilde{Y}_2,
		\end{array}\right.
	\end{equation}   
	where $\theta>0$ is a given parameter.
	Method \eqref{MCS} is of order two for any value~$\theta$.
	Here we make the common choice $\theta = \frac{1}{3}$, which is motivated by stability and accuracy results 
	for two-dimensional problems, see e.g.~\cite{HM11,HT18,HW09,HW16} and Section~\ref{SecStab}.
	It is easily verified that in \eqref{MCS} the integral part and mixed derivative term are both treated by the 
	explicit trapezoidal rule.
	We note that the explicit stages $\widehat{Y}_0$, $\widetilde{Y}_0$ can be merged, so that the integral part 
	is evaluated (just) twice per time step.
	The implicit stages $Y_j$, $\widetilde{Y}_j$ (for $j=1,2$) are often called stabilizing corrections.
	The four pertinent linear systems for these stages are tridiagonal and can be solved very efficiently 
	by means of an a priori $LU$ factorization.
	We mention that \eqref{MCS} is already applied in the first time step, i.e., for defining $V^1$ (thus 
	IMEX Euler is not used here).
	\vskip0.2cm 
	
	\item[6.] \textit{Two-step adaptation of the MCS (MCS2) scheme}:
	\vskip0.1cm
	In \cite{HT18} a second adaptation of the MCS scheme to PIDEs has been proposed where, instead of the explicit 
	trapezoidal rule, the integral part is now handled in a two-step Adams--Bashforth fashion:
	\begin{equation}\label{MCS2}
		\left\{\begin{array}{lll}
			X_0 = V^{n-1}+\Delta t A^{(D)}V^{n-1},\\
			Y_0 = X_0 + \tfrac{1}{2}\Delta t A^{(J)}(3V^{n-1}-V^{n-2}),\\
			Y_j = Y_{j-1}+\theta\Delta t A_j(Y_j-V^{n-1})\quad (j=1,2),\\
			\widehat{Y}_0 = Y_0 + \theta\Delta t A^{(M)}(Y_2-V^{n-1}),\\
			\widetilde{Y}_0 = \widehat{Y}_0 + (\tfrac{1}{2}-\theta)\Delta t A^{(D)}(Y_2-V^{n-1}),\\
			\widetilde{Y}_j = \widetilde{Y}_{j-1} + \theta \Delta t A_j(\widetilde{Y}_j-V^{n-1})\quad (j=1,2),\\
			V^n = \widetilde{Y}_2.
		\end{array}\right.
	\end{equation}   
	Method \eqref{MCS2} is also of order two for any value~$\theta$.
	We choose again $\theta = \frac{1}{3}$ and, for starting this two-step method, define $V^1$ by \eqref{MCS}.
	\vskip0.2cm 
	
	\item[7.] \textit{Stabilizing correction two-step Adams-type (SC2A) scheme}:
	\vskip0.1cm
	Hundsdorfer \& in 't Hout \cite{HH18} recently studied a novel class of stabilizing correction multistep methods 
	for the numerical solution of PDEs.
	We consider here a direct adaptation to PIDEs of a prominent member of this class, the two-step Adams-type scheme 
	called SC2A:
	\begin{equation}\label{SC2A}
		\left\{\begin{array}{lll}   
			Y_0 = V^{n-1} + \Delta t\, (A^{(M)}+A^{(J)})\sum_{i=1}^2\widehat{b}_iV^{n-i} + \Delta t\, (A_1+A_2)\sum_{i=1}^2\widecheck{b}_iV^{n-i},\\
			Y_j = Y_{j-1} + \theta \Delta t A_j(Y_j-V^{n-1})\quad (j=1,2),\\
			V^n = Y_2,
		\end{array}\right.
	\end{equation}   
	with coefficients $(\widehat{b}_1,\widehat{b}_2) = \left(\frac{3}{2},-\half\right)$ and 
	$(\widecheck{b}_1,\widecheck{b}_2) = \left(\frac{3}{2}-\theta, - \half+\theta\right)$.
	The integral part and mixed derivative term are now both handled by the two-step Adams--Bashforth scheme.
	Method \eqref{SC2A} is again of order two for any value~$\theta$. 
	Following \cite{HH18}, we take $\theta = \frac{3}{4}$, which is based upon stability and accuracy results.
	For starting \eqref{SC2A}, we define $V^1$ again by \eqref{MCS} with $\theta = \frac{1}{3}$
	\vskip0.2cm    
	
\end{enumerate}

We conclude this section with a summary of the main characteristics of the seven operator splitting schemes
in Table~\ref{summary}, indicating its order, whether it is a one- or a two-step scheme, the number of linear 
systems in each time step, and the number of evaluations of the integral part in each time step.

\begin{table}[H]
    \centering
	\caption{Main characteristics of the seven operator splitting schemes (with $l\ge 2$).}
	\vspace{0.3cm}
	\begin{tabular}{@{}l|cccc|ccc@{}}
		\toprule
		                   & CNFE & CNFI & IETR & CNAB & MCS & MCS2 & SC2A  \\
		\midrule
		 order             & 1    & $2$  & 2    & 2    & 2   & 2    & 2     \\
		 one/two-step      & 1    & 1    & 1    & 2    & 1   & 2    & 2     \\
		 linear systems    & 1    & $l$  & 1    & 1    & 4   & 4    & 2     \\
		 integral terms    & 1    & $l$  & 2    & 1    & 2   & 1    & 1     \\		
		\bottomrule
	\end{tabular}
	\label{summary}
\end{table}


\section{Stability analysis}\label{SecStab}

Let $\mu_0$, $\mu_1$, $\mu_2$, $\lambda_0$ denote any given complex numbers representing eigenvalues of the matrices 
$A^{(M)}$, $A_1$, $A_2$, $A^{(J)}$, respectively.
For the stability analysis of the operator splitting schemes from Section~\ref{SecTime} we consider the linear 
scalar test equation
\begin{equation}\label{testeqn}
V'(t) = (\mu_0+\mu_1+\mu_2+\lambda_0)\, V(t).
\end{equation}
Let $z_j = \mu_j\, \Delta t$ (for $j=0,1,2$) and $w_0 = \lambda_0\, \Delta t$ and write $w=z_0+z_1+z_2$, 
$p=(1-\theta z_1)(1-\theta z_2)$.
Application of the three IMEX schemes \eqref{CNFE}, \eqref{CNFI}, \eqref{IETR} to test equation \eqref{testeqn} 
yields a one-step linear recurrence relation of the form
\begin{equation*}
V^{n} = R(w,w_0)\, V^{n-1}\,,
\end{equation*}
where
\begin{eqnarray}
\textrm{for~\eqref{CNFE}}: &&R(w,w_0) = \label{R_CNFE}
\frac{1+\tfrac{1}{2}w+w_0}{1-\tfrac{1}{2}w}\,,\\
\textrm{for~\eqref{CNFI}}: &&R(w,w_0) = \label{R_CNFI}
\left( \frac{\tfrac{1}{2} w_0}{1-\tfrac{1}{2}w} \right)^l + \sum_{k=0}^{l-1} 
\left( \frac{\tfrac{1}{2} w_0}{1-\tfrac{1}{2}w} \right)^k \cdot 
\frac{1+\tfrac{1}{2}w+\tfrac{1}{2} w_0}{1-\tfrac{1}{2}w}\,,\\
\textrm{for~\eqref{IETR}}: &&R(w,w_0) = \label{R_IETR}
\frac{1+\tfrac{1}{2}w+w_0+\tfrac{1}{2}(w+w_0)w_0}{1-\tfrac{1}{2}w}\,.
\end{eqnarray}
Application of the IMEX scheme \eqref{CNAB} to \eqref{testeqn} yields a two-step linear recurrence relation
\begin{equation*}
V^{n} = R_1(w,w_0)\, V^{n-1}+R_0(w,w_0)\, V^{n-2}\,,
\end{equation*}
with
\begin{eqnarray}
&&R_1(w,w_0) = \label{R1_CNAB}
\frac{1+\tfrac{1}{2}w+\tfrac{3}{2}w_0}{1-\tfrac{1}{2}w}\,,\\
&&R_0(w,w_0) = \label{R0_CNAB}
\frac{-\tfrac{1}{2}w_0}{1-\tfrac{1}{2}w}\,.
\end{eqnarray}
Application of the ADI scheme \eqref{MCS} to \eqref{testeqn} gives a one-step linear recurrence of the
form
\begin{equation*}
V^{n} = R(z_0,z_1,z_2,w_0)\, V^{n-1}\,,
\end{equation*}
where
\begin{equation}\label{R_MCS}
R(z_0,z_1,z_2,w_0) = 
1+\frac{w+w_0}{p}+\theta\frac{(z_0+w_0)(w+w_0)}{p^2}+(\tfrac{1}{2}-\theta)\frac{(w+w_0)^2}{p^2}\,.
\end{equation}
Application of the two ADI schemes \eqref{MCS2}, \eqref{SC2A} to \eqref{testeqn} gives a two-step linear 
recurrence
\begin{equation*}
V^{n} = R_1(z_0,z_1,z_2,w_0)\, V^{n-1}+R_0(z_0,z_1,z_2,w_0)\, V^{n-2}\,.
\end{equation*}
Here, for \eqref{MCS2},
\begin{eqnarray}
&&R_1(z_0,z_1,z_2,w_0) = \label{R1_MCS2}
1+(w+\tfrac{3}{2}w_0)\left(\frac{1}{p}+\theta\frac{z_0}{p^2}+(\tfrac{1}{2}-\theta)\frac{w}{p^2}\right),\\
&&R_0(z_0,z_1,z_2,w_0) = \label{R0_MCS2}
-\tfrac{1}{2}w_0\left(\frac{1}{p}+\theta\frac{z_0}{p^2}+(\tfrac{1}{2}-\theta)\frac{w}{p^2}\right)\,.
\end{eqnarray}
Next, for \eqref{SC2A},
\begin{eqnarray}
&&R_1(z_0,z_1,z_2,w_0) = \label{R1_SC2A}
1+\frac{1}{p} \left( \widehat{b}_1 (z_0+w_0) + \widecheck{b}_1 (z_1+z_2) \right),\\
&&R_0(z_0,z_1,z_2,w_0) = \label{R0_SC2A}
\frac{1}{p} \left( \widehat{b}_2 (z_0+w_0) + \widecheck{b}_2 (z_1+z_2) \right) \,.
\end{eqnarray}

In the stability analysis of ADI schemes for two-dimensional convection-diffusion equations with 
mixed derivative term (and without integral term), the following condition on $z_0$, $z_1$, $z_2$ 
plays a main role:
\begin{equation}\label{stabcond}
|z_0|\le 2\sqrt{\Re z_1\Re z_2}\,,~~ \Re z_1\le 0\,,~~ \Re z_2\le 0,
\end{equation}
where $\Re$ denotes the real part of a complex number.
Condition \eqref{stabcond} has been shown to hold \cite{HW07} in the von Neumann framework where 
semidiscretization is performed by second-order central finite differences on uniform, Cartesian 
grids and periodic boundary condition.
We note that it is readily seen that \eqref{stabcond} implies $\Re w \le 0$.

The subsequent stability results are relevant to the natural situation (for finite activity jumps) 
where $|\lambda_0| T$ is of moderate size.
Our analysis relies upon a polynomial expansion in $w_0 = \lambda_0\, \Delta t$. 
Such an argument has previously been employed in \cite{HT18, KLR18} 
(cf.~also~e.g.~\cite[p.385,386]{HV03book} in a PDE context). 

For stability of a one-step linear recurrence, we consider power-boundedness of $R$, and for 
stability of a two-step linear recurrence, power-boundedness of the $2\times 2$ companion matrix 
\begin{equation*}
C
= 
\begin{pmatrix}
\, R_1 & R_0\, \\
1 & 0
\end{pmatrix}.
\end{equation*}
Denote by $\| \cdot \|$ the maximum norm for matrices.
The first result deals with the IMEX schemes \eqref{CNFE}--\eqref{CNAB}.
\begin{theorem}\label{theorem1}
Let $l\ge 1$ and $c = \sum_{k=0}^{l-1} \left(\tfrac{1}{2} |\lambda_0|T\right)^k$. Then:
\begin{itemize}
\item[\bf (a)] for the scheme \eqref{CNFE} there holds $|R(w,w_0)^n| \le e^{|\lambda_0|t_n}$,
\item[\bf (b)] for the scheme \eqref{CNFI} there holds $|R(w,w_0)^n| \le e^{c|\lambda_0|t_n}$,
\item[\bf (c)] for the scheme \eqref{IETR} there holds $|R(w,w_0)^n| \le e^{|\lambda_0|t_n}$,
\item[\bf (d)] for the scheme \eqref{CNAB} there holds $\|C(w,w_0)^n\| \le e^{2|\lambda_0|t_n}$
\end{itemize}
whenever $w\in \C$, $\Re w \le 0$.
\end{theorem}
\begin{proof} 
It is sufficient to prove the upper bounds for $n=1$.
Define $P=R(w,0)=\left(1+\tfrac{1}{2}w\right)\left(1-\tfrac{1}{2}w\right)^{-1}$ and
$Q=\tfrac{1}{2} \left(1-\tfrac{1}{2}w\right)^{-1}$.
Since $\Re w \le 0$ we have $|P|\le 1$ and $|Q|\le \tfrac{1}{2}$.

(a) The scheme \eqref{CNFE} is identical to \eqref{CNFI} with $l=1$ and the bound is given by 
part (b) with $l=1$.

(b) Let $\omega = \tfrac{1}{2} |w_0| = \tfrac{1}{2}|\lambda_0| \, \Delta t$. For \eqref{CNFI}, one can write
\begin{equation*}
R(w,w_0) = \left( w_0 Q \right)^l + \sum_{k=0}^{l-1} \left( w_0 Q \right)^k \cdot \left( P + w_0 Q \right).
\end{equation*}
It follows that
\begin{equation*}
|R(w,w_0)| \le \omega^l + \sum_{k=0}^{l-1} \omega^k (1+\omega) = 1+ 2\omega \sum_{k=0}^{l-1} \omega^k
\le 1+c|\lambda_0| \, \Delta t \le e^{c|\lambda_0| \, \Delta t}.
\end{equation*}

(c) For \eqref{IETR}, we have
\begin{equation*}
R(w,w_0) = (1+w_0)P + w_0^2 Q,
\end{equation*}
and hence 
\begin{equation*}
|R(w,w_0)| \le 1+|\lambda_0| \, \Delta t + \tfrac{1}{2}(|\lambda_0| \, \Delta t)^2 \le e^{|\lambda_0| \, \Delta t}.
\end{equation*}

(d) For \eqref{CNAB},
\begin{equation*}   
C(w,w_0) =
\begin{pmatrix}
P & 0\, \\
1 & 0
\end{pmatrix}
+
w_0 Q
\begin{pmatrix}
3 & -1 \\
0 & 0
\end{pmatrix}.
\end{equation*}
Consequently,
\begin{equation*}
\|C(w,w_0)\| \le 1+ 2|\lambda_0| \, \Delta t \le e^{2|\lambda_0| \, \Delta t}.
\end{equation*}
\end{proof}

\begin{remark}
In the case where $|\lambda_0| \, \Delta t \le c_0$ with constant $c_0\in (0,2)$,
the constant $c$ in Theorem \ref{theorem1} can be replaced by $c=(1-c_0/2)^{-1}$.
\end{remark}

For the MCS scheme, we have the following positive stability result.
\begin{theorem}\label{theorem2}
Let $c = \max\{1/\theta\,,\,2\}$.
For the scheme \eqref{MCS} there holds:
\begin{itemize}
\item[\bf (a)] If $\theta\ge \frac{1}{3}$, then $|R(z_0,z_1,z_2,w_0)^n| \le e^{c|\lambda_0|t_n}$ 
whenever $z_0, z_1, z_2 \in \R$ satisfy \eqref{stabcond},
\item[\bf (b)] If $\frac{1}{2}\le \theta \le 1$, then  $|R(z_0,z_1,z_2,w_0)^n| \le e^{c|\lambda_0|t_n}$ 
whenever $z_0, z_1, z_2 \in \C$ satisfy \eqref{stabcond}.
\end{itemize}
\end{theorem}
\begin{proof} 
It suffices again to consider $n=1$.
Define
\begin{equation*}
P = R(z_0,z_1,z_2,0) ~~\textrm{and}~~ Q = \frac{1}{p}+\frac{z_0}{p^2}+(1-\theta)\frac{z_1+z_2}{p^2}\,.
\end{equation*}
It is easily verified that
\begin{equation*}
R(z_0,z_1,z_2,w_0) = P + w_0 Q + \tfrac{1}{2} \frac{w_0^2}{p^2}\,.
\end{equation*}
By \cite[Sect.~2.1]{HW07} we immediately have that the condition \eqref{stabcond} implies 
\begin{equation*}
\left| 1+\frac{w}{p} \right| \le \left| 1 - \frac{1}{2\theta}\right| + \frac{1}{2\theta}\,.
\end{equation*}
Next, it is readily seen that $|p|\ge 1$ and $|p|\ge \theta |z_1+z_2|$.
Hence,
\begin{equation*}
|Q| \le
\left| 1+\frac{z_0}{p}+(1-\theta)\frac{z_1+z_2}{p} \right| \le
\left| 1+\frac{w}{p} \right| + \theta \left| \frac{z_1+z_2}{p} \right|\le
\left| 1 - \frac{1}{2\theta}\right| + \frac{1}{2\theta} + 1 = c.
\end{equation*}
Consequently, if $|P|\le 1$, then we obtain
\begin{equation*}
| R(z_0,z_1,z_2,w_0) | \le 
1+ c|\lambda_0| \, \Delta t + \tfrac{1}{2}(|\lambda_0| \, \Delta t)^2 \le e^{c |\lambda_0| \, \Delta t} \,.
\end{equation*}
Parts (a) and (b) now directly follow by invoking \cite[Thm.~2.5]{HW09} and \cite[Thm.~2.7]{HM11}, 
respectively, which provide sufficient conditions for $|P|\le 1$ under \eqref{stabcond}.
\end{proof}

In a similar fashion, the subsequent stability result has been obtained in \cite[Thm.~3.7]{HT18}
for the MCS2 scheme.
\begin{theorem}\label{theorem3}
Let $c = \max\{1/\theta\,,\,2\}$. For the scheme \eqref{MCS2} there holds:
\begin{itemize}
\item[\bf (a)] If $\theta\ge \frac{1}{3}$, then $\|C(z_0,z_1,z_2,w_0)^n\|\le e^{2c|\lambda_0|t_n}$ 
whenever $z_0, z_1, z_2 \in \R$ satisfy \eqref{stabcond},
\item[\bf (b)] If $\frac{1}{2}\le \theta \le 1$, then $\|C(z_0,z_1,z_2,w_0)^n\|\le e^{2c|\lambda_0|t_n}$
 whenever $z_0, z_1, z_2 \in \C$ satisfy \eqref{stabcond}.
\end{itemize}
\end{theorem}

In Theorems \ref{theorem2} and \ref{theorem3} above, the parts (a) and (b) are relevant to, respectively, 
diffusion-dominated and convection-dominated equations.

\begin{remark}
In \cite[Thm.~3.3]{M16} it has been proved that for any given $\theta \in [\tfrac{1}{4},\tfrac{1}{2})$ 
there exists a value $\gamma \in [\tfrac{1}{2},1)$ such that $|R(z_0,z_1,z_2,0)| \le 1$ for all complex 
numbers $z_0, z_1, z_2$ satisfying the natural condition
\begin{equation}\label{stabcond2}
|z_0|\le 2\gamma \sqrt{\Re z_1\Re z_2}\,,~~ \Re z_1\le 0\,,~~ \Re z_2\le 0.
\end{equation}
The quantity $\gamma$ can be viewed as a bound on the relative size of the mixed derivative coefficient, 
that is, on $|\rho|$.
For the pertinent $\gamma$, the stability bounds in Theorems \ref{theorem2} and \ref{theorem3} for the 
schemes \eqref{MCS} and \eqref{MCS2} also hold whenever $z_0,z_1,z_2\in\C$ satisfy \eqref{stabcond2}.
In the special case $\theta =\tfrac{1}{3}$ the result from \cite{M16} yields that
$\gamma \ge (2+\sqrt{10})/6 \approx 0.86$ and it has been conjectured in \cite{HM11} that the optimal
value $\gamma \approx 0.96$.
Hence, in this case, \eqref{stabcond2} forms only a slightly stronger condition than \eqref{stabcond}.
\end{remark}

For the SC2A scheme \eqref{SC2A}, deriving a stability bound of the kind in Theorems \ref{theorem1}(d) 
and \ref{theorem3} is more complicated, among others due to the fact that the implicit method for the 
PDE part is a not a one-step but a two-step method. We shall leave this topic for future research.
We mention, however, that for \eqref{SC2A} the positive result has been obtained in \cite[Thm.~4.2]{HH18}
that if $\theta \ge \tfrac{2}{3}$, then for any given $z_0, z_1, z_2 \in \R$ satisfying \eqref{stabcond} 
the pertinent matrix $C(z_0,z_1,z_2,0)$ is power-bounded.

The stability results relevant to the scalar test equation \eqref{testeqn} can directly be employed in 
a stability analysis of the splitting schemes for two-dimensional PIDEs with constant convection and 
diffusion coefficients and with the integral term representing a two-dimensional cross-correlation
(or convolution). 
It is well-known that the transformation to the log-price variable $x_i = \ln(s_i)$ (for $i=1,2$) turns
\eqref{PIDE2D} with general probability density function $f$ into such a PIDE, where the spatial domain 
is $\R^2$.
The obtained probability density function in our case is that of the double exponential distribution.
Semidiscretization on a uniform Cartesian grid on the whole $\R^2$ domain with second-order central finite 
differences for convection and diffusion and a standard discretization (employing bilinear interpolation as 
in Subsection~\ref{SecIntpart}) of the double integral then yields a linear system of ODEs with doubly 
infinite matrix (operator) $A$ that is of block Laurent type. 
Its constituent matrices $A^{(M)}$, $A_1$, $A_2$, $A^{(J)}$ are also of this type, and stability can
be analyzed in the $l_2$-norm by means of a Fourier transform, 
cf.~\cite[Sect.~3.4]{KLR18}.
In particular, the (scaled) eigenvalues of these matrices, which are given through their symbols, are 
readily seen to satisfy \eqref{stabcond} and $|\lambda_0| \le \lambda$.
Consequently, the stability theorems of this section can be used to arrive at positive, unconditional 
stability results for the operator splitting schemes of Section~\ref{SecTime} when applied to the PIDE 
for the log-price variable. 
For the sake of brevity, we shall omit the details here.

Additional, complementary stability results relevant to (one- and two-dimensional) PIDEs have been derived 
in \cite{CF08} for the scheme \eqref{CNFI}, in \cite{ST14} for the scheme \eqref{CNAB} and in \cite{HT18} 
for the schemes \eqref{MCS} and \eqref{MCS2}.

	
\section{Numerical study}\label{SecResults}
In this section we present ample numerical experiments for the seven operator splitting schemes formulated in 
Section~\ref{SecTime}, which provides important insight in their convergence behavior and mutual performance.
For the CNFI scheme \eqref{CNFI} we choose $l=2$ iterations, as noted in Section~\ref{SecTime}.
The experiments involve three different parameter sets for the two-asset Kou model, labelled as 1, 2, and~3. 
The parameter set~1 is taken from Clift \& Forsyth \cite{CF08}. 
Set~2 is a blend of parameters considered for the one-asset Kou model by Almendral \& Oosterlee~\cite{AO05} and d'Halluin, Forsyth 
\& Vetzal~\cite{HFV05} (see also Toivanen~\cite{T08}). 
Set~3 has been newly constructed and includes a relatively large value for the product $\lambda T$, i.e., the expected number of 
jumps in $[0,T]$. 
For the truncated spatial domain, we (heuristically) select $S_{\rm max} = 20K, 10K, 30K$ for sets 1, 2, 3, respectively.

\begin{table}[H]
    \centering
	\caption{Parameter sets for the two-asset Kou jump-diffusion model.}
	\vspace{0.3cm}
	\begin{tabular}{@{}llllllllllllll@{}}
		\toprule
		& $\sigma_1$ & $\sigma_2$ & $r$ & $\rho$ & $\lambda$ & $p_1$ & $p_2$ & $\eta_{p_1}$ & $\eta_{q_1}$ & $\eta_{p_2}$ & $\eta_{q_2}$ & $K$ & $T$ \\ \midrule
		Set 1 & 0.12 & 0.15 & 0.05 & 0.30 & 0.50 & 0.40 & 0.60 & $1/0.20$ & $1/0.15$ & $1/0.18$ & $1/0.14$ & 100 & 1 \\
		Set 2 & 0.15 & 0.20 & 0.05 & 0.50 & 0.20 & 0.3445 & 0.50 & 3.0465 & 3.0775 & 3 & 2 & 100 & 0.2 \\
		Set 3 & 0.20 & 0.30 & 0.05 & 0.70 & 8 & 0.60 & 0.65 & 5 & 4 & 4 & 3 & 100 & 1 \\ \bottomrule
	\end{tabular}
	\label{paramsets}
\end{table}

\subsection{Option values and numerical convergence behavior}\label{SecResults1}
Figure~\ref{FigSurfacePlots} shows the numerically computed\footnote{Using $m_1=m_2=200$, $N=100$ and the MCS2 scheme.} 
option value surfaces for the European put-on-the-average option and the three parameter sets from Table~\ref{paramsets} 
on the asset price domain $[0,3K]\times[0,3K]$ and $t=T$.

\begin{figure}[H]
	\centering
	\includegraphics[trim = 1.4in 3.7in 1.3in 3.7in, clip, scale=0.5]{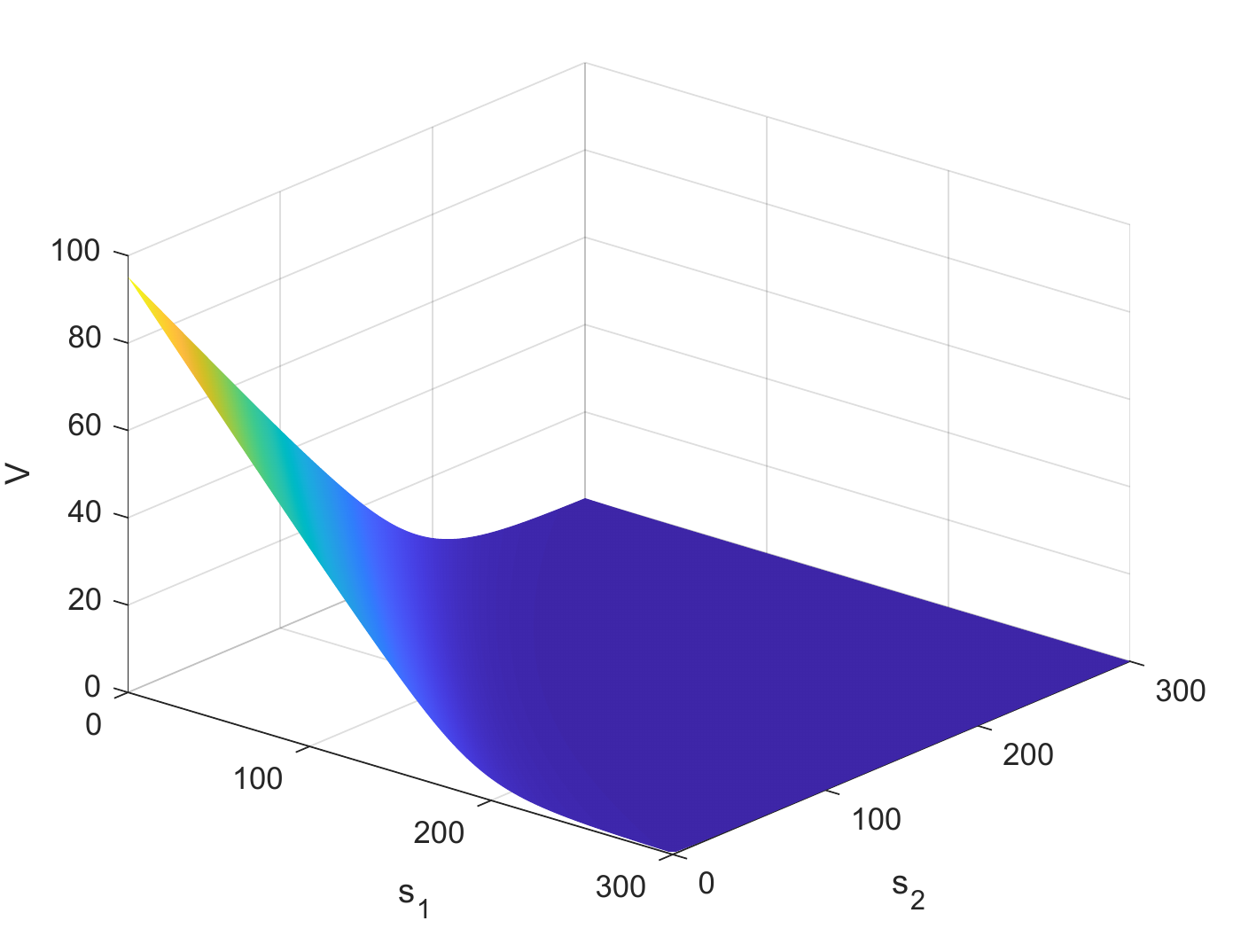}	
	\includegraphics[trim = 1.4in 3.7in 1.3in 3.7in, clip, scale=0.5]{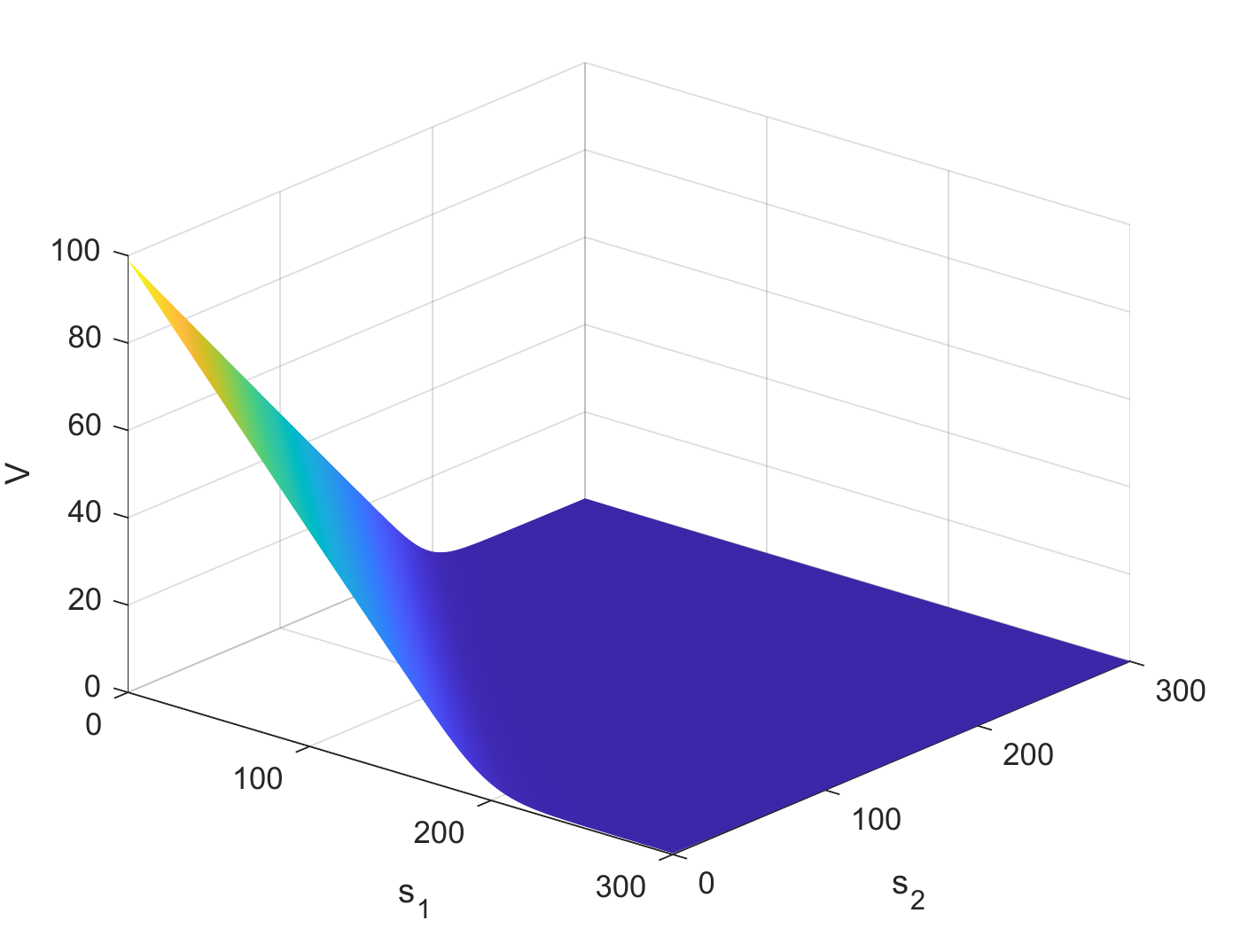}	
	\includegraphics[trim = 1.3in 3.7in 1.3in 3.7in, clip, scale=0.5]{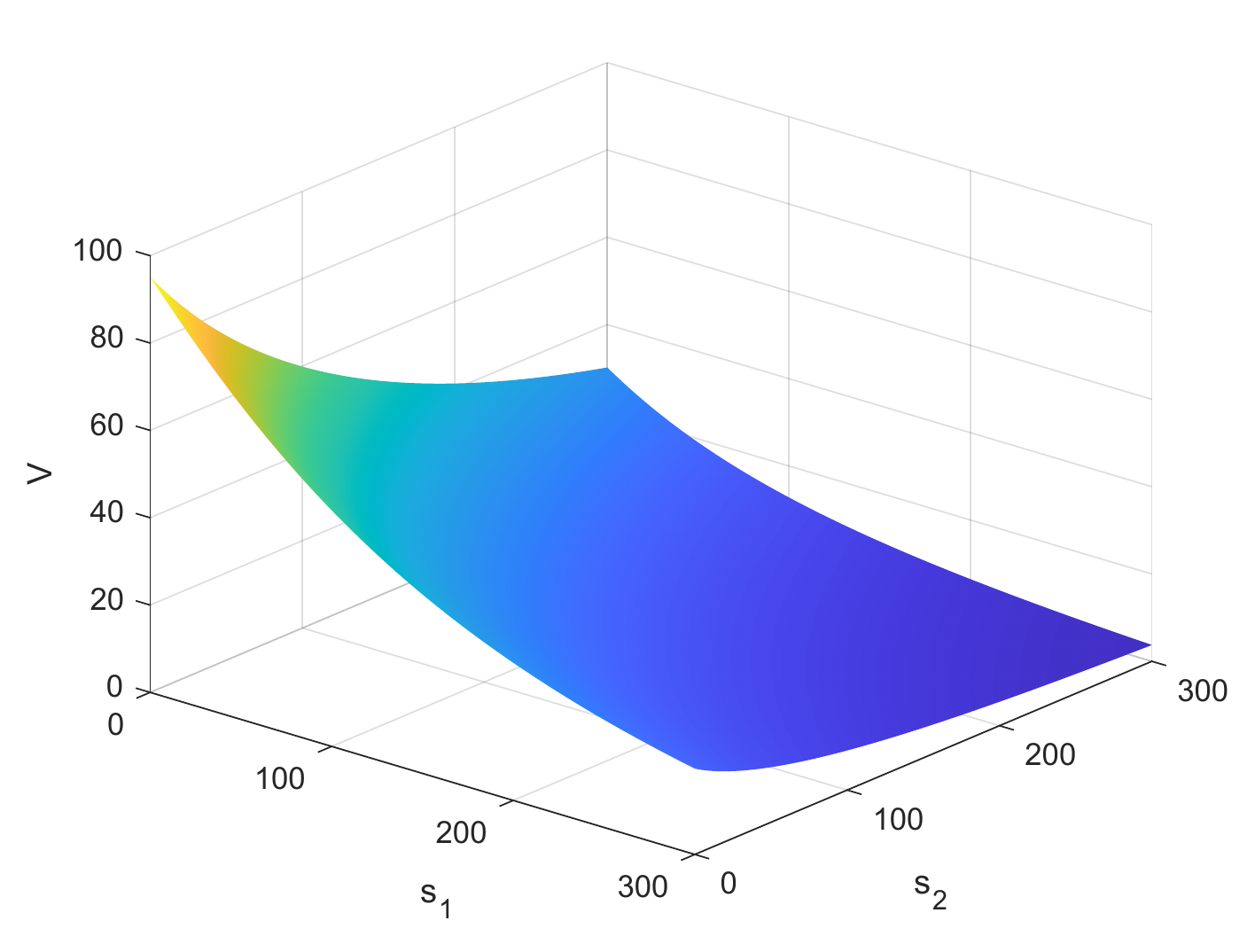}
	\caption{Option value surfaces for the European put-on-the-average option under the two-asset Kou model in the case of parameter 
	set 1 (top left), 2 (top right) and 3 (bottom) given in Table~\ref{paramsets}.}
	\label{FigSurfacePlots}
\end{figure}

\noindent
For future reference, accurate approximations to the option values have been computed\footnote{Using $m_1=m_2=1000$, $N=500$ 
and the MCS2 scheme with spline interpolation.} for specifically chosen spot prices $S_0^{(1)}$, $S_0^{(2)}$ of the two assets 
in a neighborhood of the strike $K$, see Table~\ref{tablevalues}. 
\vskip0.3cm

\begin{table} [H]
    \centering
	\caption{\label{tablevalues}European put-on-the-average option value approximations for parameter sets 1, 2 and 3.}
	\vspace{0.3cm}
	Set 1\\\vspace{0.1cm}
	\begin{tabular}{llrrrrr}
		\toprule
		&  & $S_0^{(1)} = 90$ &  & $S_0^{(1)} = 100$ &  & $S_0^{(1)} = 110$ \\
		&  & \multicolumn{1}{l}{} & \multicolumn{1}{l}{} & \multicolumn{1}{l}{} & \multicolumn{1}{l}{} & \multicolumn{1}{l}{} \\
		$S_0^{(2)} = 90$  &  & 8.9385 &  & 6.0316 &  & 3.8757 \\
		$S_0^{(2)} = 100$ &  & 5.9655 &  & 3.8038 &  & 2.3370 \\
		$S_0^{(2)} = 110$ &  & 3.7641 &  & 2.2978 &  & 1.3771 \\ 
		\bottomrule
	\end{tabular}\vspace{0.5cm}

	Set 2\\\vspace{0.1cm}
	\begin{tabular}{llrrrrr}
		\toprule
		&  & $S_0^{(1)} = 90$ &  & $S_0^{(1)} = 100$ &  & $S_0^{(1)} = 110$ \\
		&  & \multicolumn{1}{l}{} & \multicolumn{1}{l}{} & \multicolumn{1}{l}{} & \multicolumn{1}{l}{} & \multicolumn{1}{l}{} \\
		$S_0^{(2)} = 90$  &  & 9.6863 &  & 5.5616 &  & 2.6400 \\
		$S_0^{(2)} = 100$ &  & 5.6162 &  & 2.6929 &  & 1.1264 \\
		$S_0^{(2)} = 110$ &  & 2.7570 &  & 1.1670 &  & 0.5246 \\ 
		\bottomrule
	\end{tabular}\vspace{0.5cm}

	Set 3\\\vspace{0.1cm}
	\begin{tabular}{llrrrrr}
		\toprule
		&  & $S_0^{(1)} = 90$ &  & $S_0^{(1)} = 100$ &  & $S_0^{(1)} = 110$ \\
		&  & \multicolumn{1}{l}{} & \multicolumn{1}{l}{} & \multicolumn{1}{l}{} & \multicolumn{1}{l}{} & \multicolumn{1}{l}{} \\
		$S_0^{(2)} = 90$  &  & 32.7459 &  & 31.0984 &  & 29.5758 \\
		$S_0^{(2)} = 100$ &  & 30.5796 &  & 29.0181 &  & 27.5770 \\
		$S_0^{(2)} = 110$ &  & 28.5830 &  & 27.1033 &  & 25.7396 \\ 
		\bottomrule
	\end{tabular}\vspace{0.5cm}
\end{table}

We next examine, through numerical experiments, the convergence behavior of the seven operator splitting schemes formulated 
in Section~\ref{SecTime} for the temporal discretization of the semidiscrete two-dimensional Kou PIDE. 
Take $m_1=m_2=m$.
Consider a region of financial interest
\begin{equation*}\label{ROI}
	\textrm{ROI} = \left\{(s_1,s_2) : \tfrac{1}{2}K < s_1,s_2 < \tfrac{3}{2}K\right\}
\end{equation*}
and define the {\it temporal discretization error} on this ROI for $t=T$ by
\begin{equation}\label{EROI}
	\widehat{E}^{\rm ROI}(m,N) = \textrm{max} \left\{| V_{i,j}(T) - V_{i,j}^{N'} |: \Delta t = T/N' ~~{\rm and}~~ 
	(s_{1,i},s_{2,j})\in \rm ROI\right\}.
\end{equation}
Clearly, this error is measured in the (important) maximum norm.
The MCS2 scheme is applied with $3000$ time steps to obtain a reference value for the exact semidiscrete solution $V(T)$ to 
\eqref{ODEs}.
For each of the seven operator splitting schemes formulated in Section~\ref{SecTime}, we study the temporal discretization 
error for a sequence of values $N$, where the approximation $V^{N'}$ to $V(T)$ is computed using $N'$ time steps.
Here $N'$ is chosen in function of $N$ depending on the specific scheme, so as to arrive at a fair comparison between the 
ADI schemes and between the IMEX schemes.

For the three ADI schemes we find in our experiments that the evaluation of the two-dimensional integral part, by means of 
the algorithm from Subsection~\ref{SecIntpart}, takes the same computational time as the solution of four to six pertinent 
tridiagonal linear systems (using an a priori $LU$ factorization). 
The scheme (\ref{MCS}) employs two evaluations of the integral part, whereas (\ref{MCS2}) and (\ref{SC2A}) each use only one 
such evaluation.
Next, both schemes (\ref{MCS}) and (\ref{MCS2}) require the solution of four tridiagonal systems, whereas (\ref{SC2A}) 
contains only two tridiagonal systems.
Accordingly, for a fair mutual comparison, we apply the ADI schemes (\ref{MCS}), (\ref{MCS2}) and (\ref{SC2A}) with, respectively, 
$N'=N$, $N' = \lfloor 3N/2 \rfloor$ and $N' = 2N$ time steps.

For the four IMEX schemes, a proper choice of $N'$ depends on the method used for solving the pertinent linear systems involving 
the matrix $I-\tfrac{1}{2}\Delta t A^{(D)}$.
Recall that this matrix is sparse, but has a large bandwidth.
Hence, due to fill-in, the direct solution of these linear systems by $LU$ factorization becomes prohibitively expensive 
when $m$ gets large.
Below we shall consider CPU times in the case of the BiCGSTAB iterative method. 
It turns out that in this case a sound comparison is achieved by applying the two schemes (\ref{CNFE}) and (\ref{CNAB}) with 
$N' = 2N$ time steps, the scheme (\ref{IETR}) with $N' = \lfloor 3N/2 \rfloor$ time steps, and the scheme (\ref{CNFI}) with 
just $N'=N$ time steps.
Notice that in the latter scheme (with $l=2$) two linear systems arise per time step, whereas in each of the other three IMEX 
schemes it is only one.

For each of the seven splitting schemes, the temporal discretization errors $\widehat{E}^{\rm ROI}(m,N)$ have been computed for 
$m = 200$ and a sequence of values $N$ between 10 and 1000. 
Here a direct solver for the linear systems has been used also in the case of the four IMEX schemes, to ensure that the approximation 
errors due to the iterative solver do not affect the outcome.
Figure~\ref{FigTemperrors} displays the obtained temporal errors for all three parameter sets given in Table~\ref{paramsets}, 
where the left column shows the results for the IMEX schemes and the right column the results for the ADI schemes.

From Figure~\ref{FigTemperrors} we observe the positive result that, for each given scheme and parameter set, the temporal errors 
are bounded from above by a moderate value and decrease monotonically as $N$ increases. 
As expected, the CNFE scheme (\ref{CNFE}) shows an order of convergence equal to one, whereas all other six splitting schemes 
reveal a desirable order of convergence equal to two.
By repeating the numerical experiments for spatial grids that are finer (e.g.~$m=300,400,500$) or coarser (e.g.~$m=50, 100$) we 
obtain results that are visually identical to those displayed in Figure~\ref{FigTemperrors}. 
This indicates that, for each splitting scheme, the temporal errors are essentially independent of the spatial grid 
size, and hence, its observed convergence order is valid in a stiff sense, which is very favorable.

It is interesting to remark that the error constants become larger as the intensity of the jumps increases, keeping all
else fixed.
This has been observed in additional numerical experiments for increasing values of $\lambda$ between 0 and 10.

As already alluded to above, we give CPU times in the case where the BiCGSTAB method\footnote{As implemented in Matlab version 
R2020b through the function {\tt bicgstab}.} is employed for the iterative solution of the pertinent linear systems in the four 
IMEX schemes.
Here an ILU preconditioner is used, which has been computed upfront.
As a natural starting vector for the iteration we have taken $V^{n-1}$ for the schemes (\ref{CNFE}), (\ref{IETR}), (\ref{CNAB}) 
and $Y_{k-1}$ for the scheme (\ref{CNFI}).
The tolerance\footnote{For the relative residual error in the 2-norm.} tol = 1e-10 is heuristically chosen small enough so that 
the approximation error due to the iterative solution of the linear systems in each time step does not appear to strongly affect 
the temporal convergence behavior of the IMEX schemes, as considered in Figure~\ref{FigTemperrors}.

Table~\ref{cpuIMEXandADI} shows the obtained individual CPU times (in seconds) for the seven operator splitting schemes, 
applied with $N'$ time steps, in the numerical solution of the two-dimensional Kou PIDE \eqref{PIDE2D} for parameter set~1 
and a range of values $m_1=m_2=m$ between 100 and 1000 with $N=m/2$.
The algorithm of Subsection~\ref{SecIntpart} has been employed to approximate the integral \eqref{integral}.

A perusal of the results of Table~\ref{cpuIMEXandADI} indicates that, for each splitting method, the CPU time is 
essentially directly proportional to $Nm^2$, which is as desired.
Clearly, for each given $m$, the obtained CPU times for the CNFE, CNFI and CNAB schemes (using $N'$ time steps) 
are almost identical and for the IETR scheme it is either equal to this or somewhat larger.
Next, for each given $m$, the CPU times for the MCS, MCS2 and SC2A schemes are almost the same and about half
of those obtained for the IMEX schemes.

Among the four IMEX schemes, the CNAB scheme yields the smallest error constants in our numerical experiments,
cf.~Figure~\ref{FigTemperrors}.
Among the three ADI schemes, the MCS2 scheme has the smallest error constant in all our experiments.
Overall, the MCS2 scheme is preferred.
It yields temporal errors that are smaller than or approximately equal to those of the CNAB scheme, but is 
computationally significantly faster.
Also, the linear systems in each time step of the MCS2 scheme are just tridiagonal, and can therefore be 
solved exactly in a highly efficient manner.

\begin{figure}[H]
	\centering
	\includegraphics[trim = 1.3in 3.7in 1in 3.7in, clip, scale=0.5]{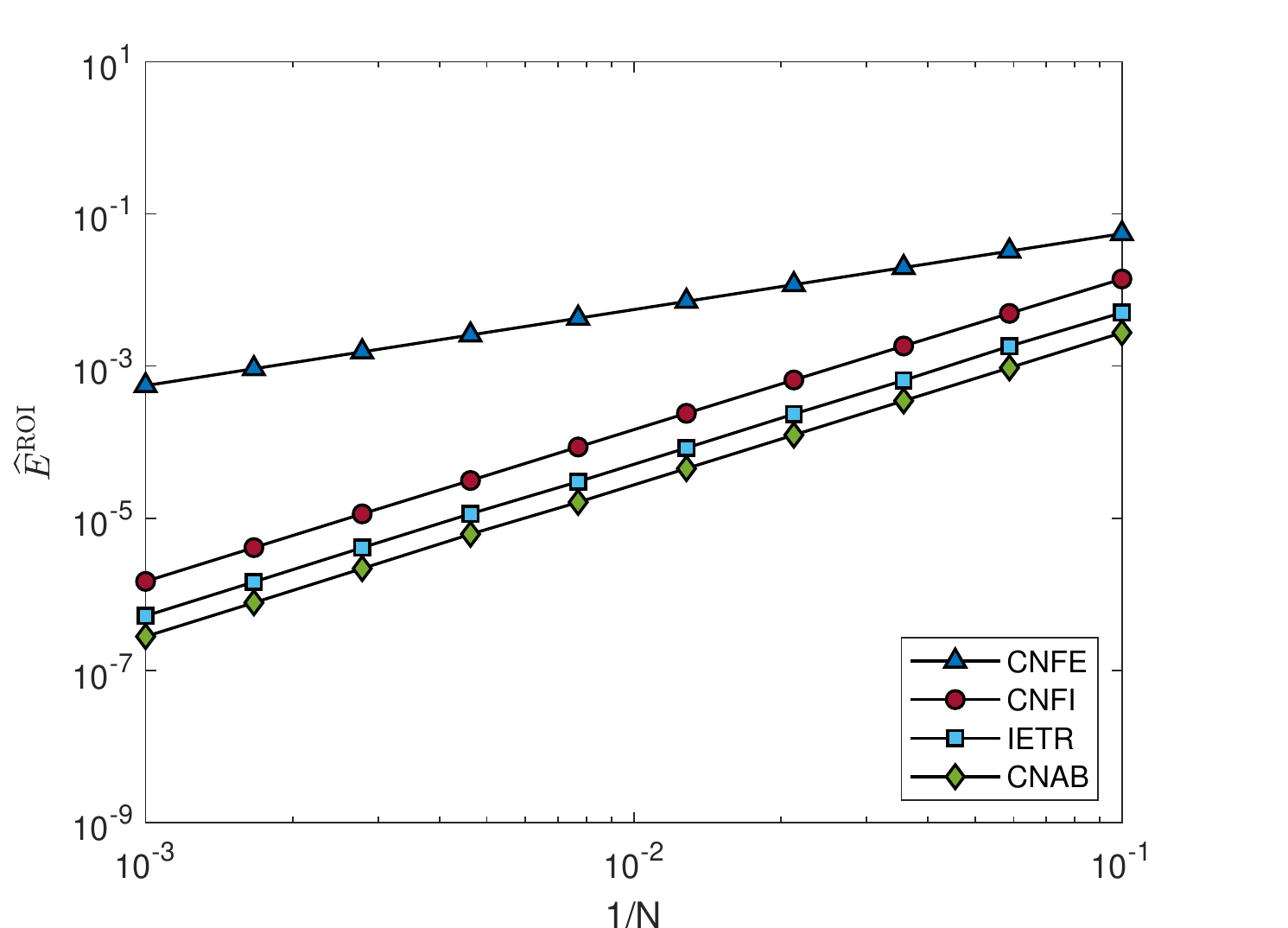}\includegraphics[trim = 1.3in 3.7in 1.6in 3.7in, clip, scale=0.5]{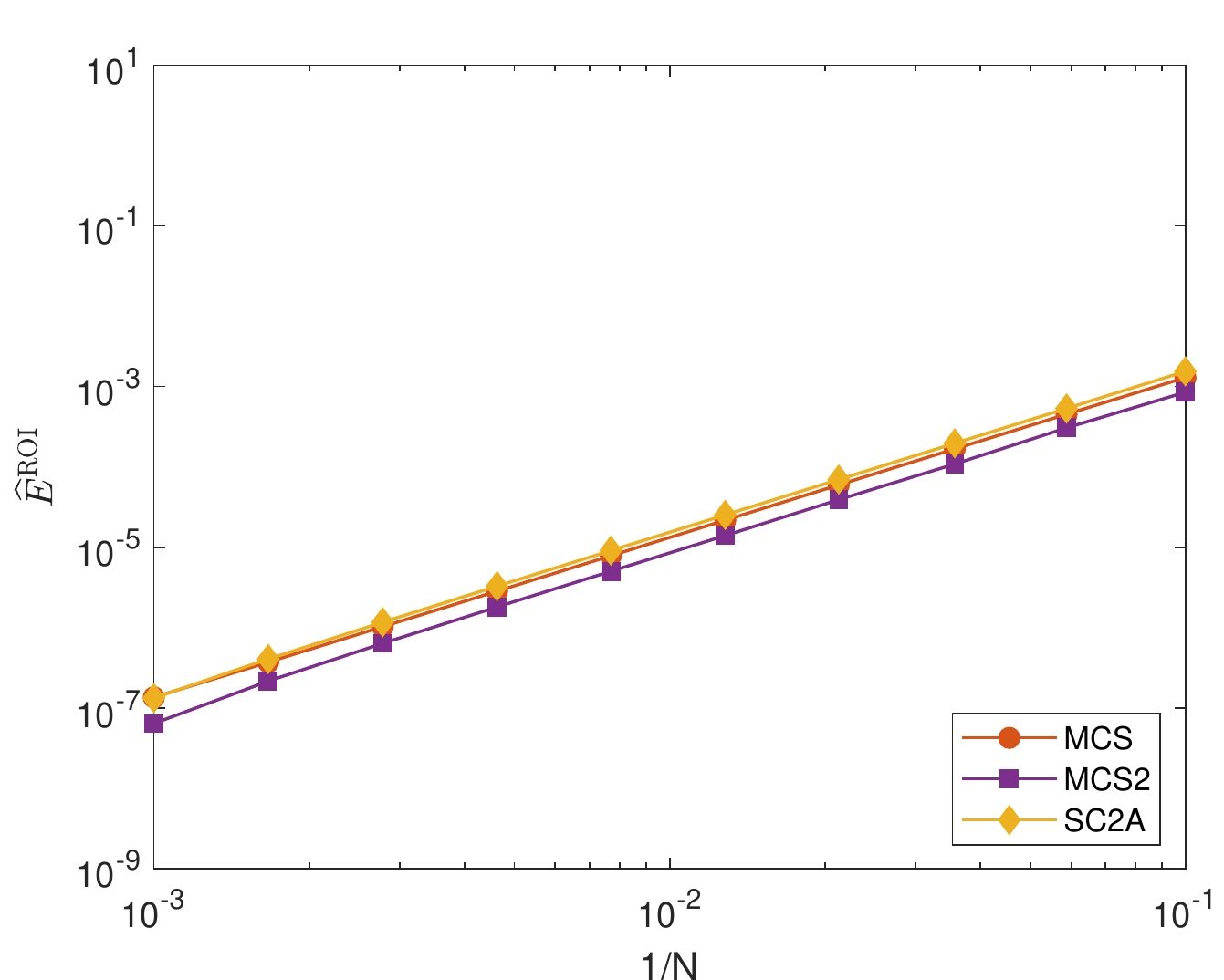}\\
	\includegraphics[trim = 1.3in 3.7in 1in 3.7in, clip, scale=0.5]{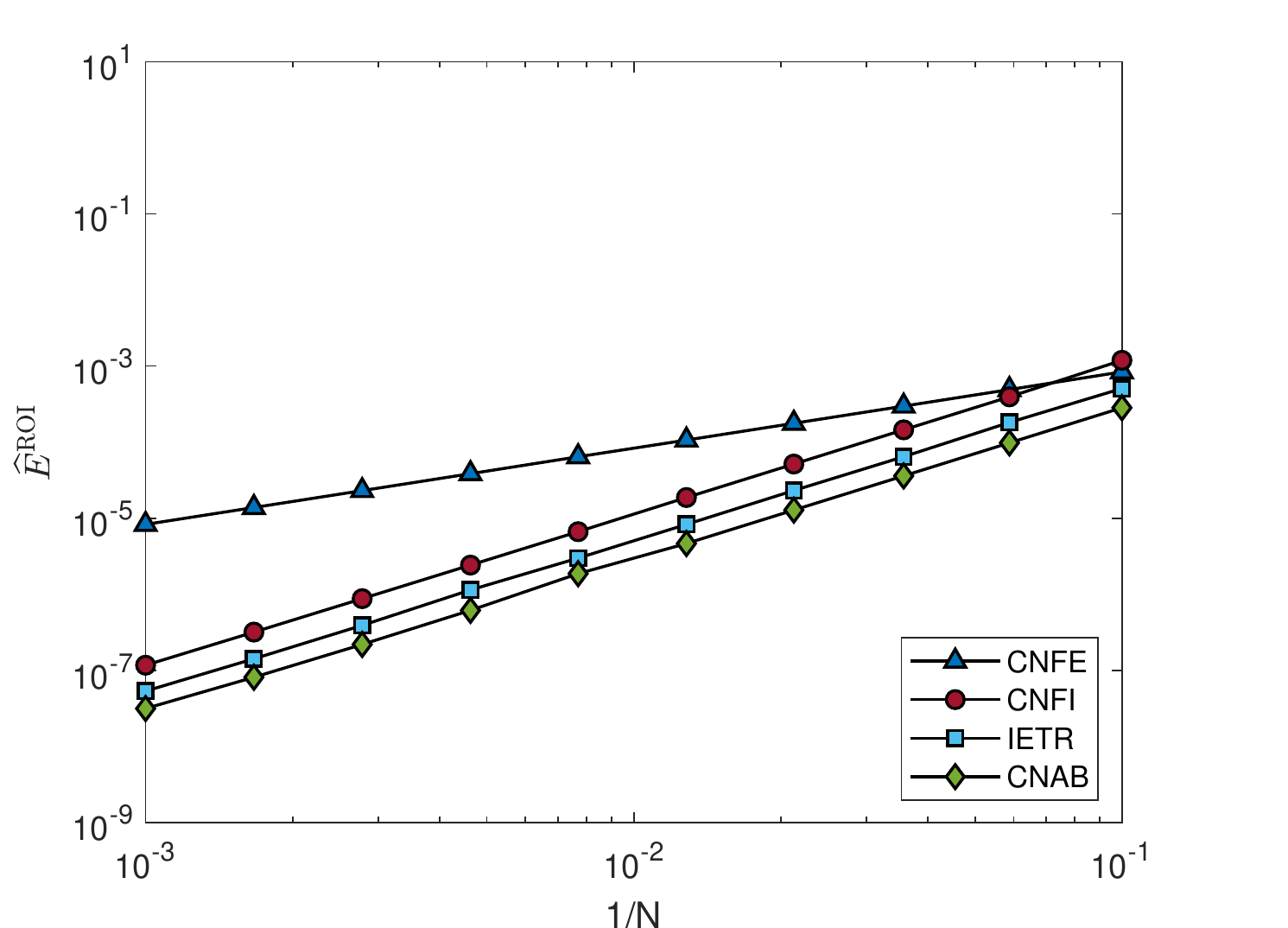}\includegraphics[trim = 1.3in 3.7in 1.6in 3.7in, clip, scale=0.5]{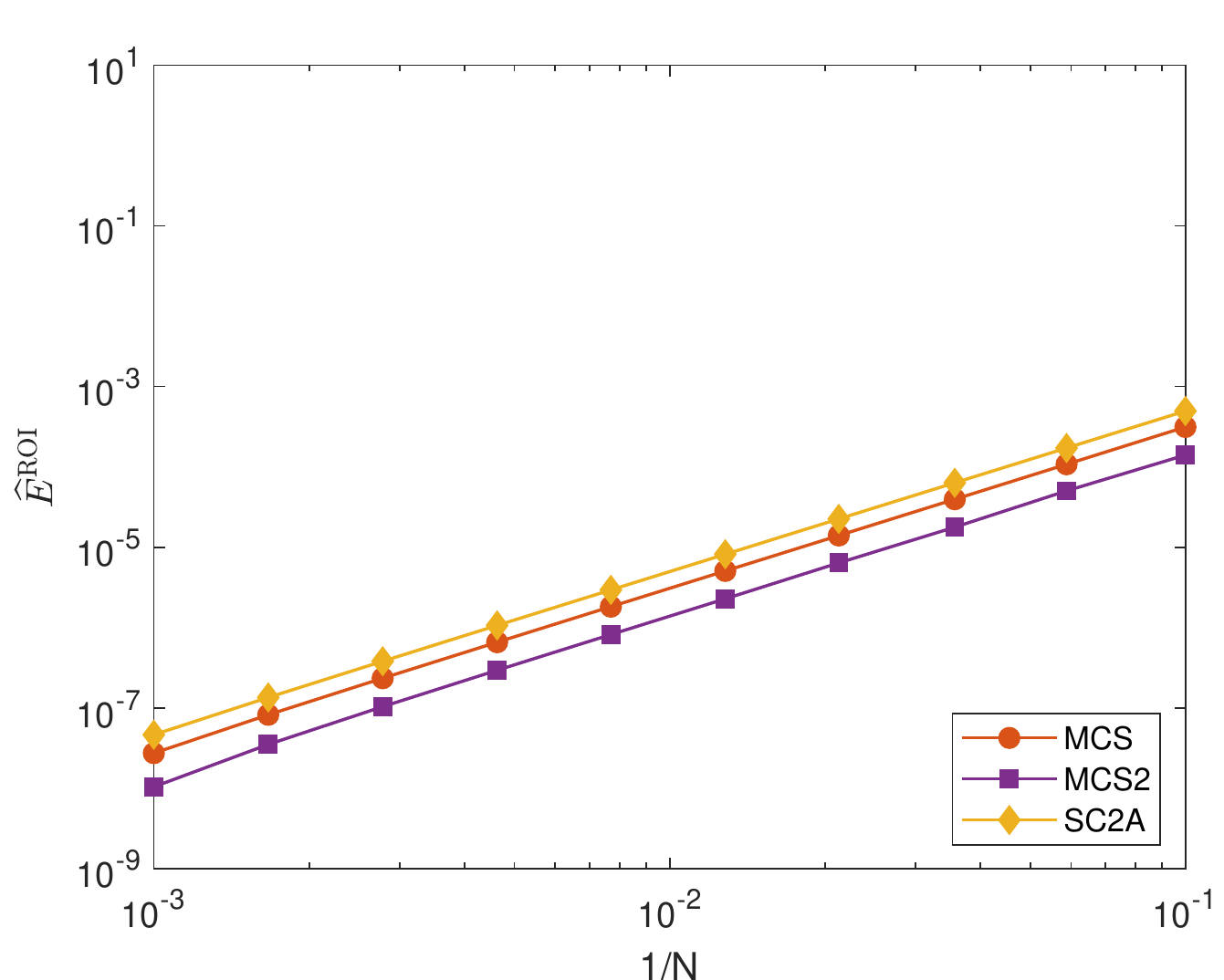}\\
	\includegraphics[trim = 1.3in 3.7in 1in 3.7in, clip, scale=0.5]{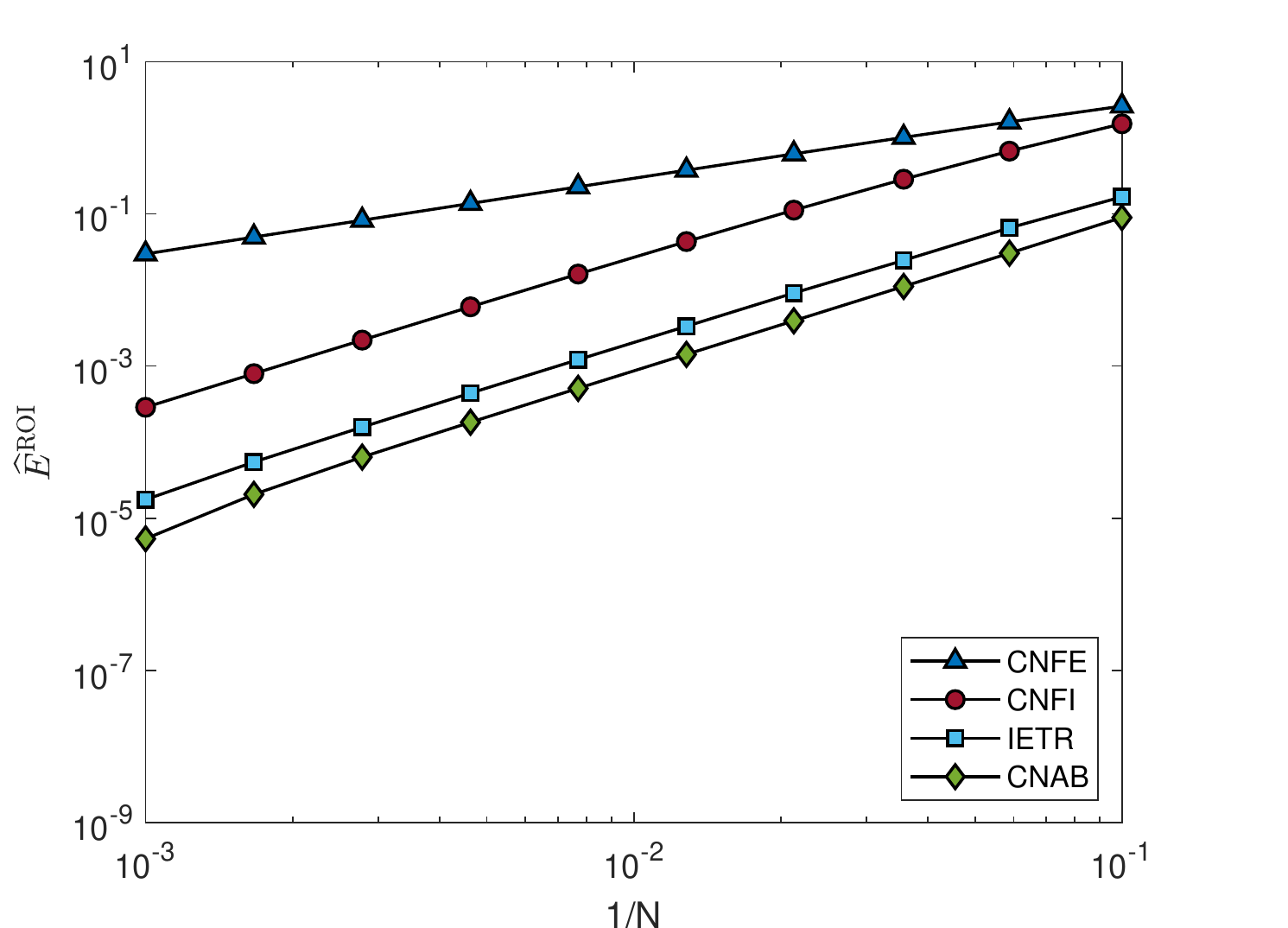}\includegraphics[trim = 1.3in 3.7in 1.6in 3.7in, clip, scale=0.5]{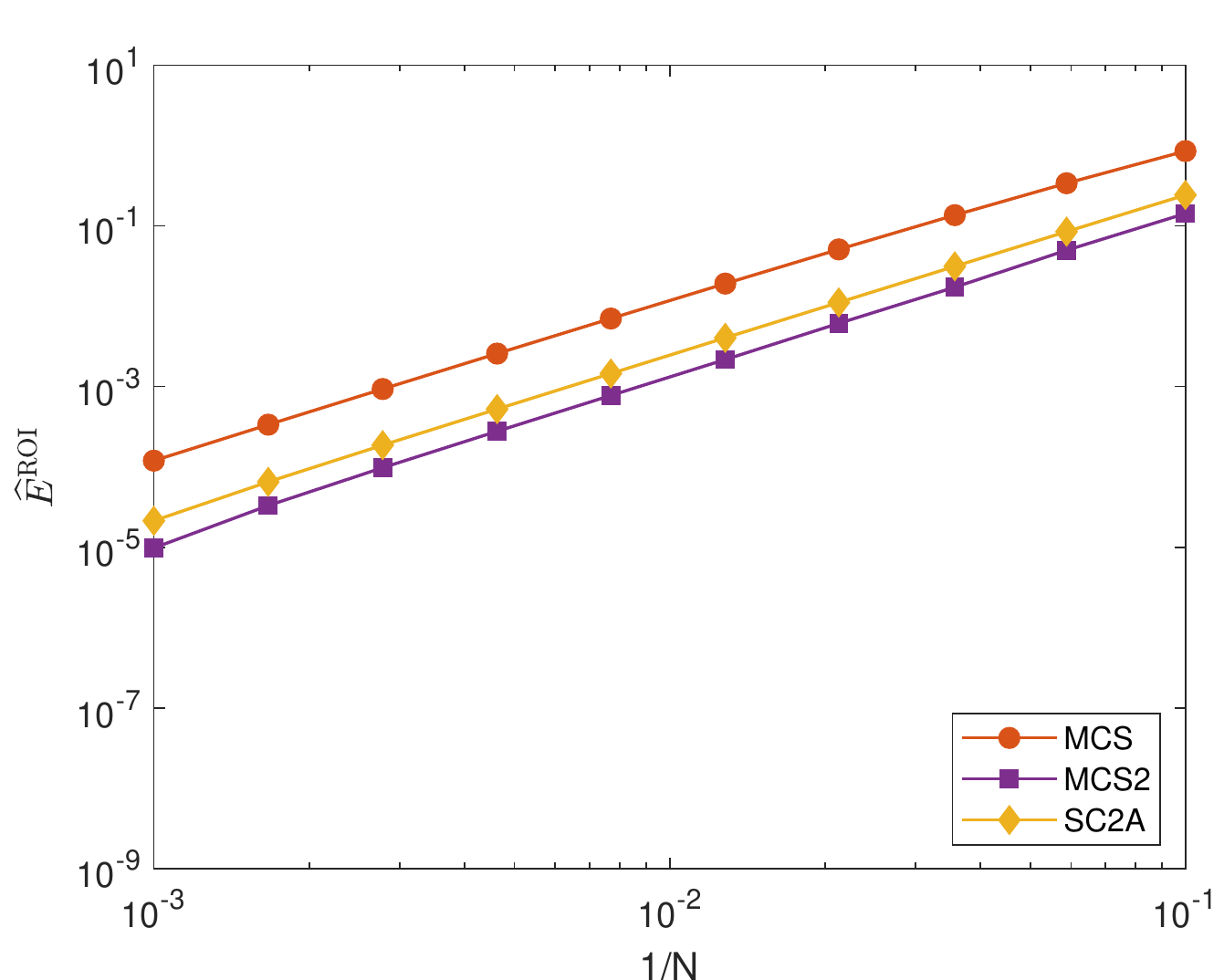}
	\caption{Temporal errors $\widehat{E}^{\rm ROI}(200,N)$ of the seven operator splitting schemes under consideration. IMEX schemes on the left side and ADI 
	schemes on the right side, with parameter set 1 (top), set 2 (middle) and set 3 (bottom) from Table~\ref{paramsets}. The CNFI and MCS schemes are applied 
	with step size $\Delta t = T/N$, the IETR and MCS2 schemes are applied with $\Delta t = T/\lfloor 3N/2 \rfloor$ and the CNFE, 
	CNAB, SC2A schemes are applied with $\Delta t = T/(2N)$.}
	\label{FigTemperrors}
\end{figure}

\begin{table}[H]
    \centering
	\caption{CPU times (s) for the seven operator splitting schemes applied with $N'$ time steps 
	(given between brackets) in the numerical solution of \eqref{PIDE2D} for parameter set 1 with $m_1=m_2=m$ 
	and $N=m/2$.}
	\vspace{0.3cm}
	\begin{tabular}{@{}rr|rrrr|rrr@{}}
		\toprule
		$m$ & $N$ & CNFE   & CNFI & IETR     & CNAB   & MCS   & MCS2     & SC2A   \\
		    &     & ($2N$) &($N$) & ($3N\!/2$) & ($2N$) & ($N$) & ($3N\!/2$) & ($2N$) \\		
		\midrule
		100 &  50 &  0.4 &    0.4 &   0.4 &   0.4 &   0.2 &   0.2 &   0.2 \\
		200 & 100 &  2.3  &   2.0 &   2.1 &   2.3 &   1.1 &   1.1 &   1.1 \\
		300 & 150 &  8.8  &   8.5 &   8.8 &   9.0 &   4.4 &   4.4 &   4.3 \\
		400 & 200 &  23.8 &  24.1 &  25.4 &  24.6 &  13.6 &  13.5 &  14.0 \\ 
		500 & 250 &  49.1 &  49.2 &  56.8 &  49.2 &  27.0 &  27.0 &  28.4 \\
	   1000 & 500 & 503.2 & 499.7 & 555.2 & 501.0 & 247.0 & 240.6 & 253.3 \\ 		
		\bottomrule
	\end{tabular}
	\label{cpuIMEXandADI}
\end{table}

\subsection{The Greeks}\label{Greeks}
The Greeks are mathematical derivatives of the option value with respect to underlying variables and parameters. 
They constitute a measure for risk that indicates how sensitive an option value is to changes in underlying variables and parameters 
and are crucial for hedging strategies. 
In this subsection we consider the numerical approximation of the Delta and Gamma Greeks. 
Delta is a measure for the rate of change of the option value with respect to a change in an underlying asset price.
As there are two underlying assets, there are two Deltas:
\begin{equation*}
\Delta_1 = \frac{\partial v}{\partial s_1} \quad \text{and} \quad \Delta_2 = \frac{\partial v}{\partial s_2} . 
\end{equation*}
Next, Gamma measures the rate of change of a Delta with respect to a change in an underlying asset price. 
There are three different Gammas:
\begin{equation*}
	\Gamma_{11} = \frac{\partial \Delta_1}{\partial s_1} = \frac{\partial^2v}{\partial s_1^2}, \quad 
	\Gamma_{22} = \frac{\partial \Delta_2}{\partial s_2} = \frac{\partial^2v}{\partial s_2^2} \quad \text{and} \quad 
	\Gamma_{12} = \frac{\partial \Delta_1}{\partial s_2} = \frac{\partial^2v}{\partial s_2 \partial s_1} = \frac{\partial \Delta_2}{\partial s_1} = \Gamma_{21}.
\end{equation*}

By virtue of the finite difference discretization that has been defined in Section \ref{SecSpatial}, the Delta and Gamma Greeks
can directly be approximated, at essentially no computational cost, by applying the second-order central finite 
difference formulas considered in Subsection \ref{SecPDEpart} to the option value approximations on the spatial grid.

As an illustration, Figure~\ref{FigGreeks} displays the numerically\footnote{Using $m_1=m_2=200$, $N=100$ and the MCS2 scheme.} 
obtained Delta and Gamma surfaces at maturity for the European put-on-the-average option under the two-asset Kou model for 
parameter set~1. 
As expected, the Delta surfaces are steepest around the line segment $s_1 + s_2 = 2K$ and, correspondingly, the 
Gamma surfaces are highest there.

Similarly to the option value, we study the temporal convergence behavior of all operator splitting schemes in 
the case of the five Greeks.
Akin to (\ref{EROI}), the temporal discretization error in the case of Delta $\Delta_k \, (k = 1, 2)$ is defined by
\begin{equation}\label{EROIGreeks}
	\widehat{E}^{\rm ROI}_{\Delta_k}(m,N) = \textrm{max} \left\{| (\Delta_k)_{i,j}(T) - (\Delta_k)_{i,j}^{N'} |: 
	\Delta t = T/N' ~~{\rm and}~~ (s_{1,i},s_{2,j})\in \rm ROI \right\}.
\end{equation}
Here $\Delta_k (T)$ denotes the pertinent finite difference matrix for convection applied to the reference value for the exact
semidiscrete solution $V(T)$ given in Subsection~\ref{SecResults1}. 
Next, $\Delta_k ^{N'}$ is equal to the same finite difference matrix applied to the approximation $V^{N'}$ of $V(T)$ that is 
generated by any one of the seven operator splitting schemes. 
Analogous temporal discretization error definitions hold in the case of $\Gamma_{11}, \Gamma_{22}$, $\Gamma_{12}$. 

Figure~\ref{FigGreeksError} displays the temporal errors in the case of the five Greeks and parameter set 1 for $m = 200$ and 
the same sequence of values $N$ between 10 and 1000 as before.
We arrive at the same conclusions on the temporal convergence behavior of the seven splitting schemes as obtained 
in Subsection~\ref{SecResults1} regarding the option value.
In particular, besides CNFE, all splitting schemes reveal a stiff order of convergence equal to two, and MCS2 has the best
performance among all these schemes.

We note that the somewhat larger errors that are observed for the CNFI, IETR, CNAB schemes when $N$ is small are attributed
to the nonsmoothness of the initial (payoff) function and could be alleviated by applying four (instead of two) half time steps
with IMEX Euler at the start of the time stepping, cf. Section~\ref{SecTime}.

\begin{figure}[H]
	\centering
	\includegraphics[trim = 1.3in 3.7in 1in 3.7in, clip, scale=0.5]{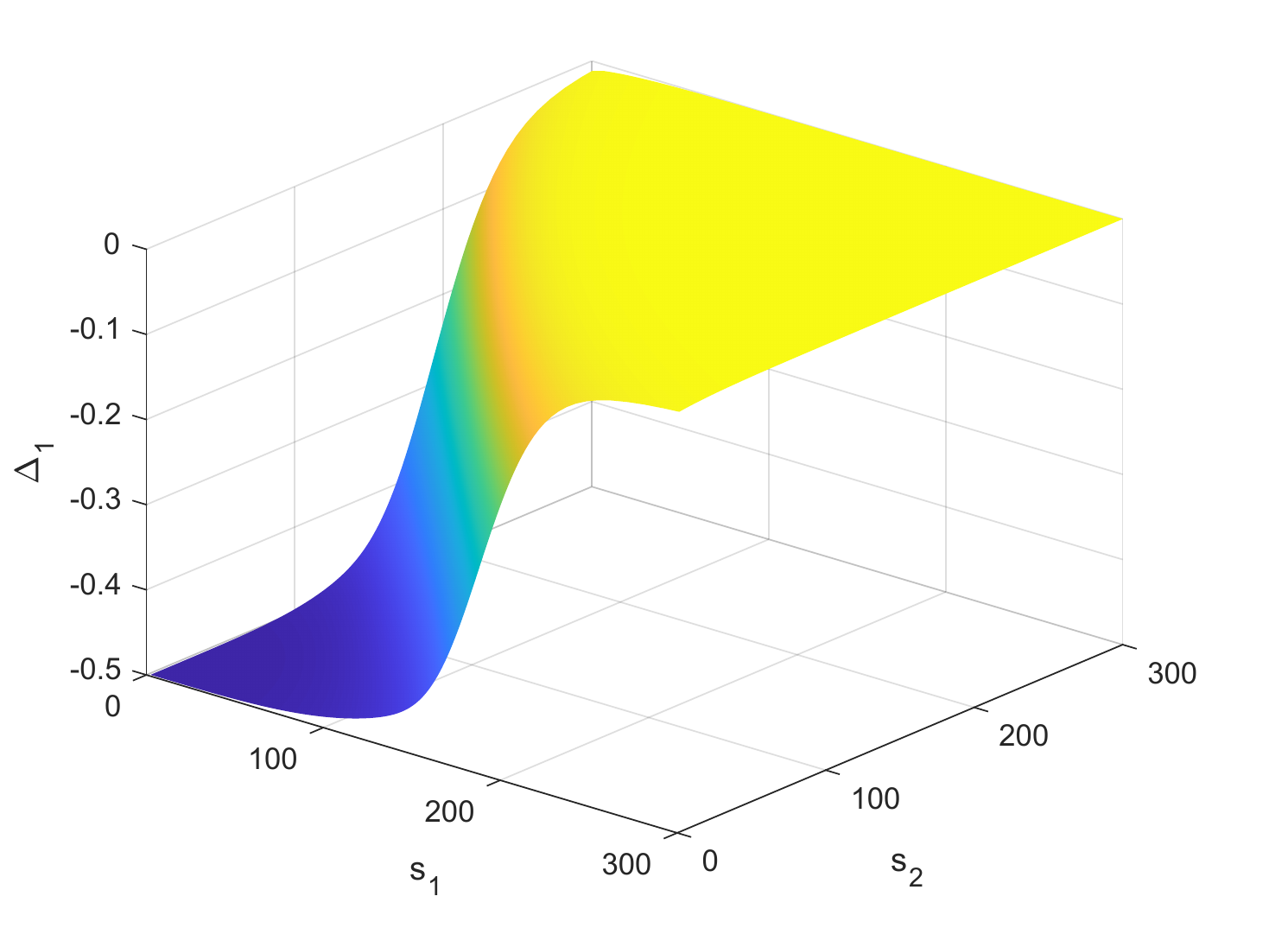}\includegraphics[trim = 1.3in 3.7in 1.4in 3.7in, clip, scale=0.5]{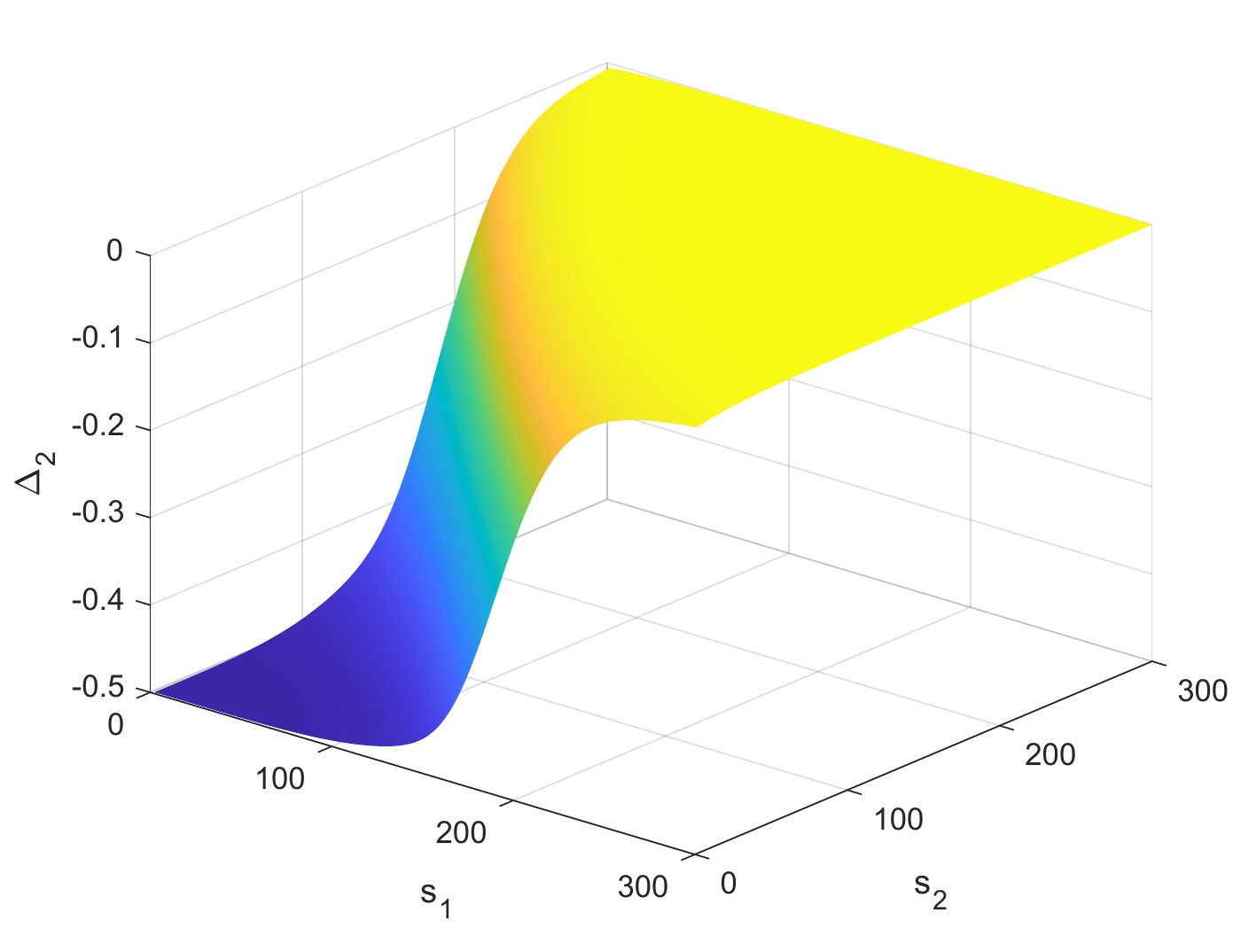}\\
	\includegraphics[trim = 1.3in 3.7in 1in 3.7in, clip, scale=0.5]{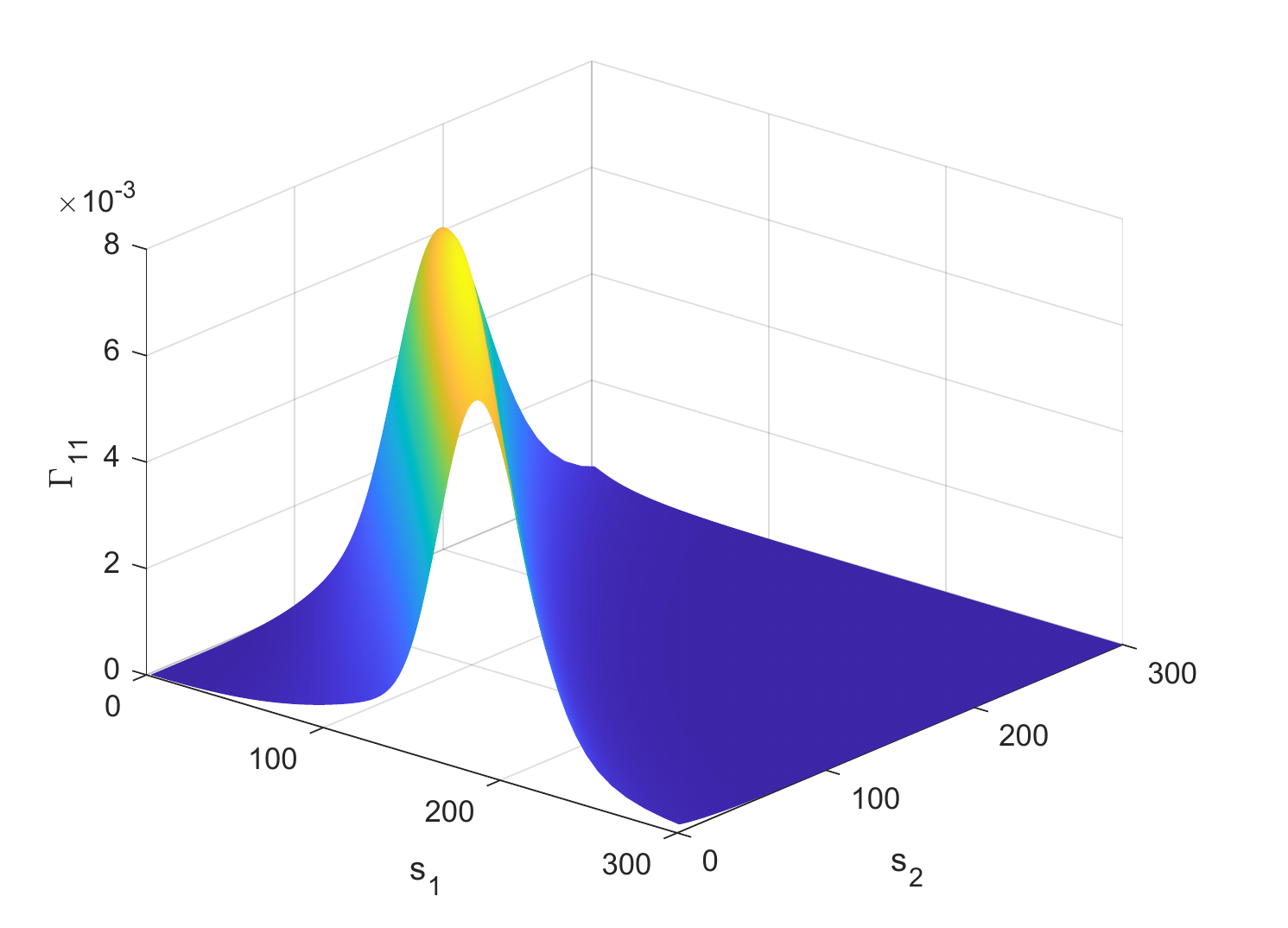}\includegraphics[trim = 1.3in 3.7in 1.4in 3.7in, clip, scale=0.5]{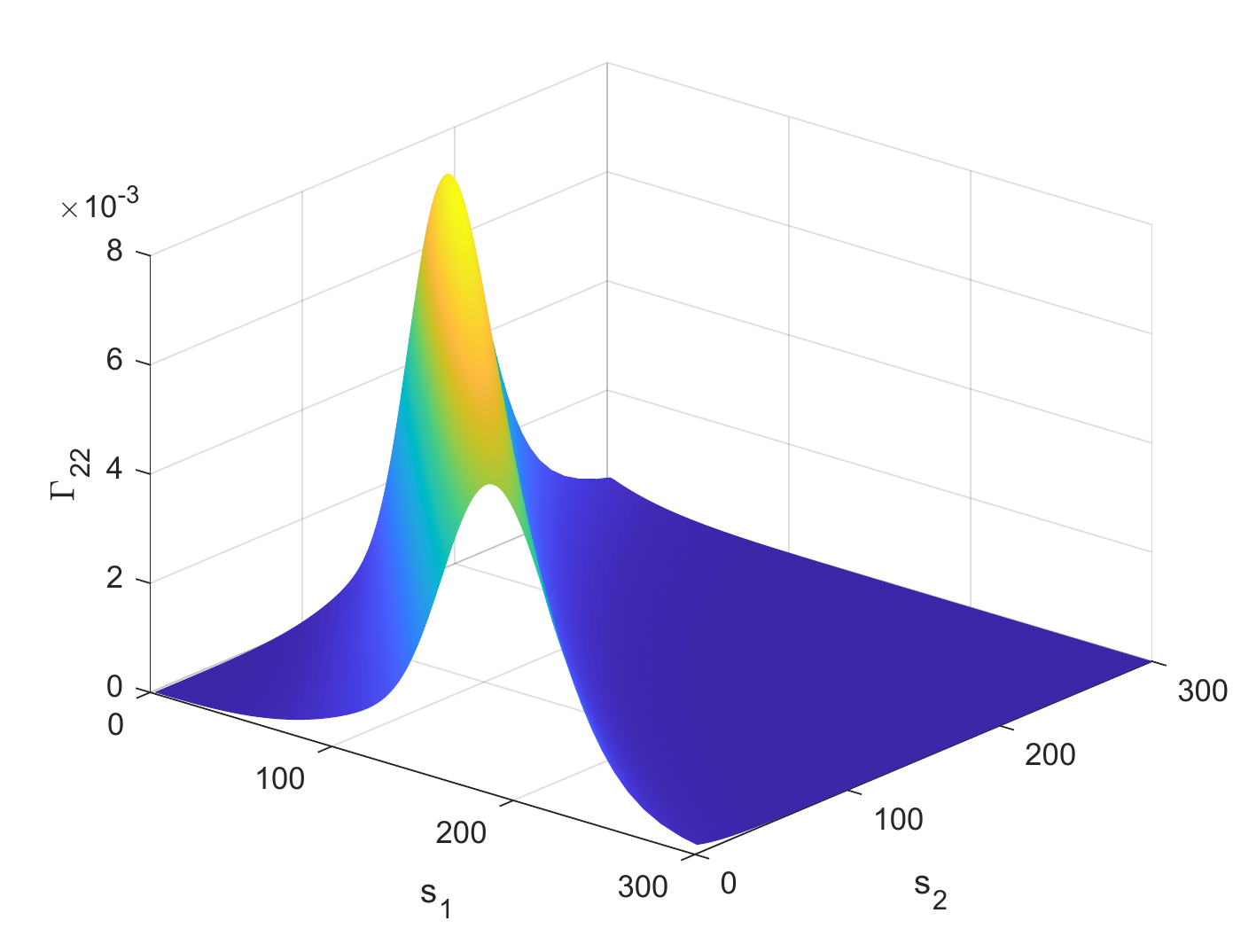}\\
	\includegraphics[trim = 1.3in 3.7in 1in 3.7in, clip, scale=0.5]{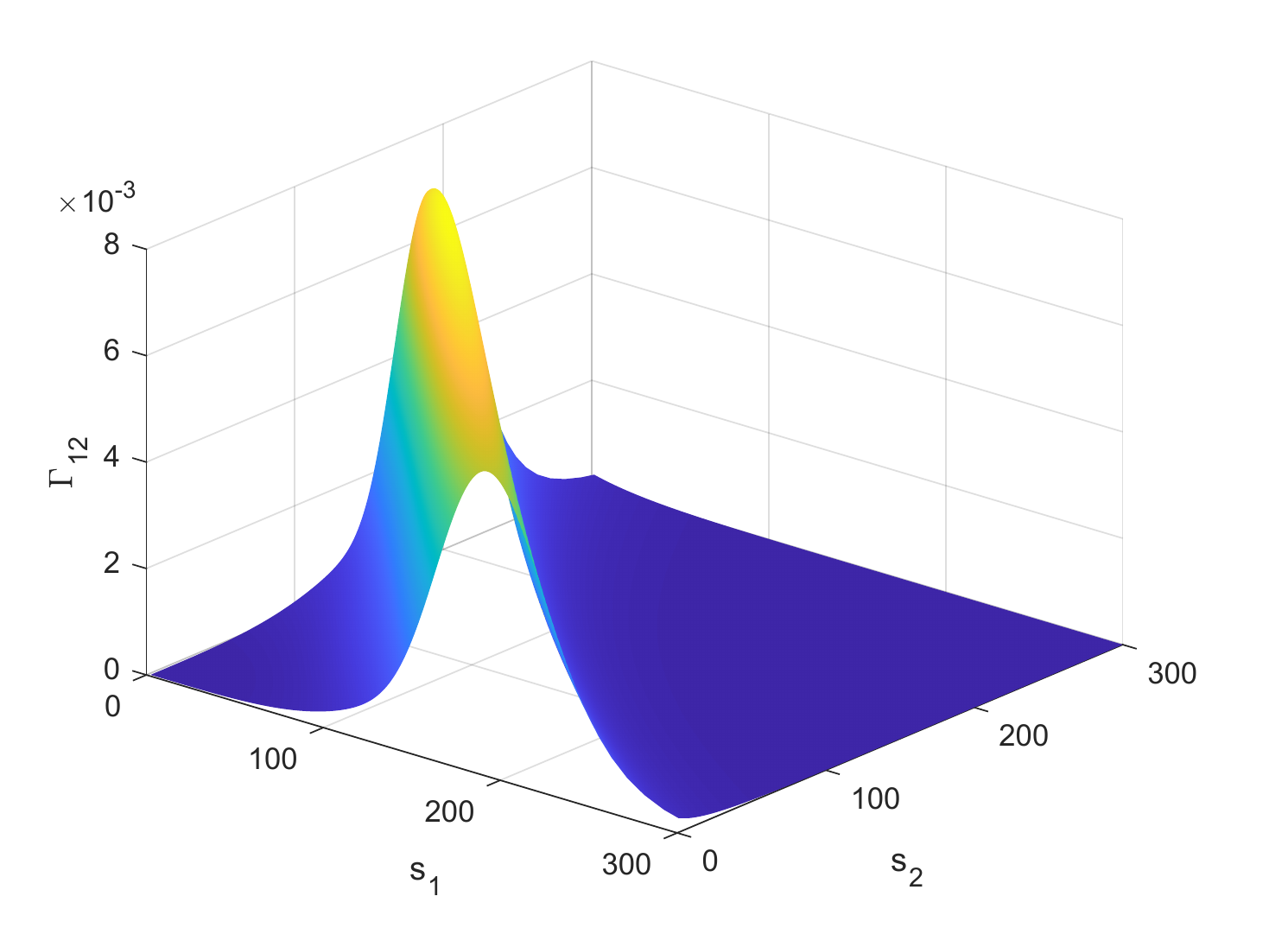}
	\caption{
	First-order Greeks $\Delta_1$ (top left) and $\Delta_2$ (top right) and second-order Greeks $\Gamma_{11}$ (middle left), $\Gamma_{22}$ 
	(middle right) and $\Gamma_{12}$ (bottom) for the European put-on-the-average option under the two-asset Kou model in the case of 
	parameter set 1.}
	\label{FigGreeks}
\end{figure}

\begin{figure}[H]
	\centering
	\includegraphics[trim = 1.3in 3.7in 1in 3.7in, clip, scale=0.5]{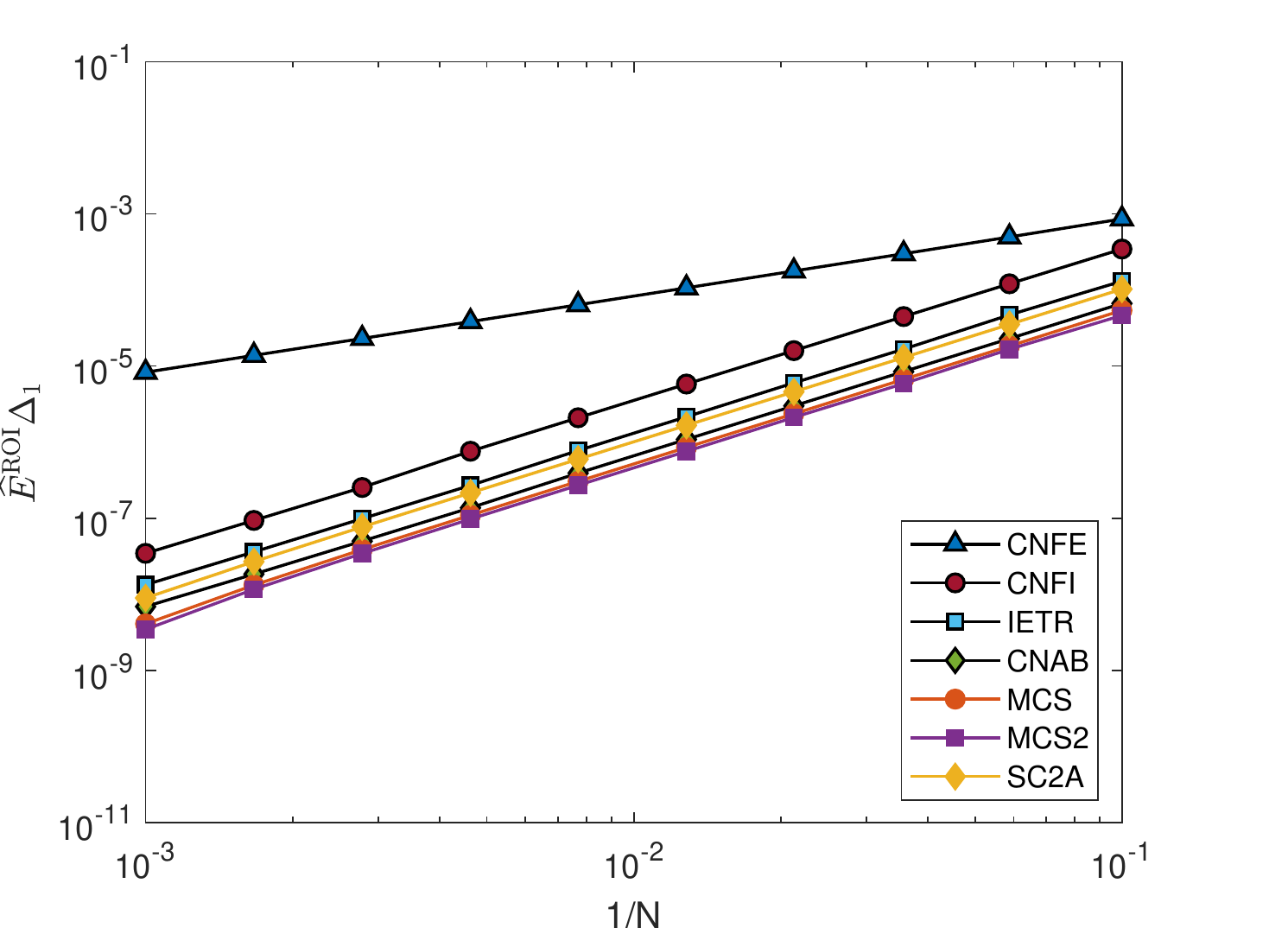}\includegraphics[trim = 1.3in 3.7in 1.6in 3.7in, clip, scale=0.5]{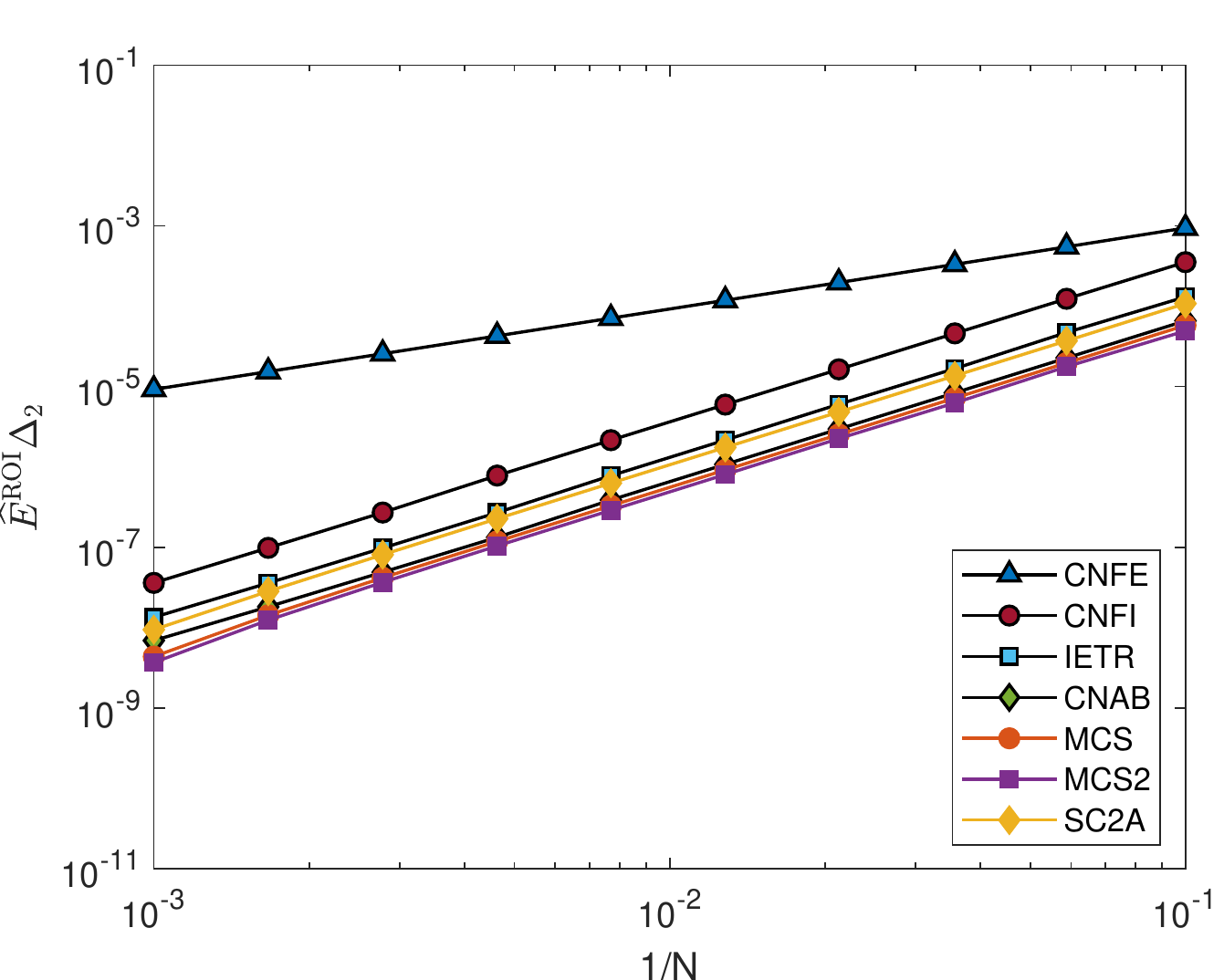}\\
	\includegraphics[trim = 1.3in 3.7in 1in 3.7in, clip, scale=0.5]{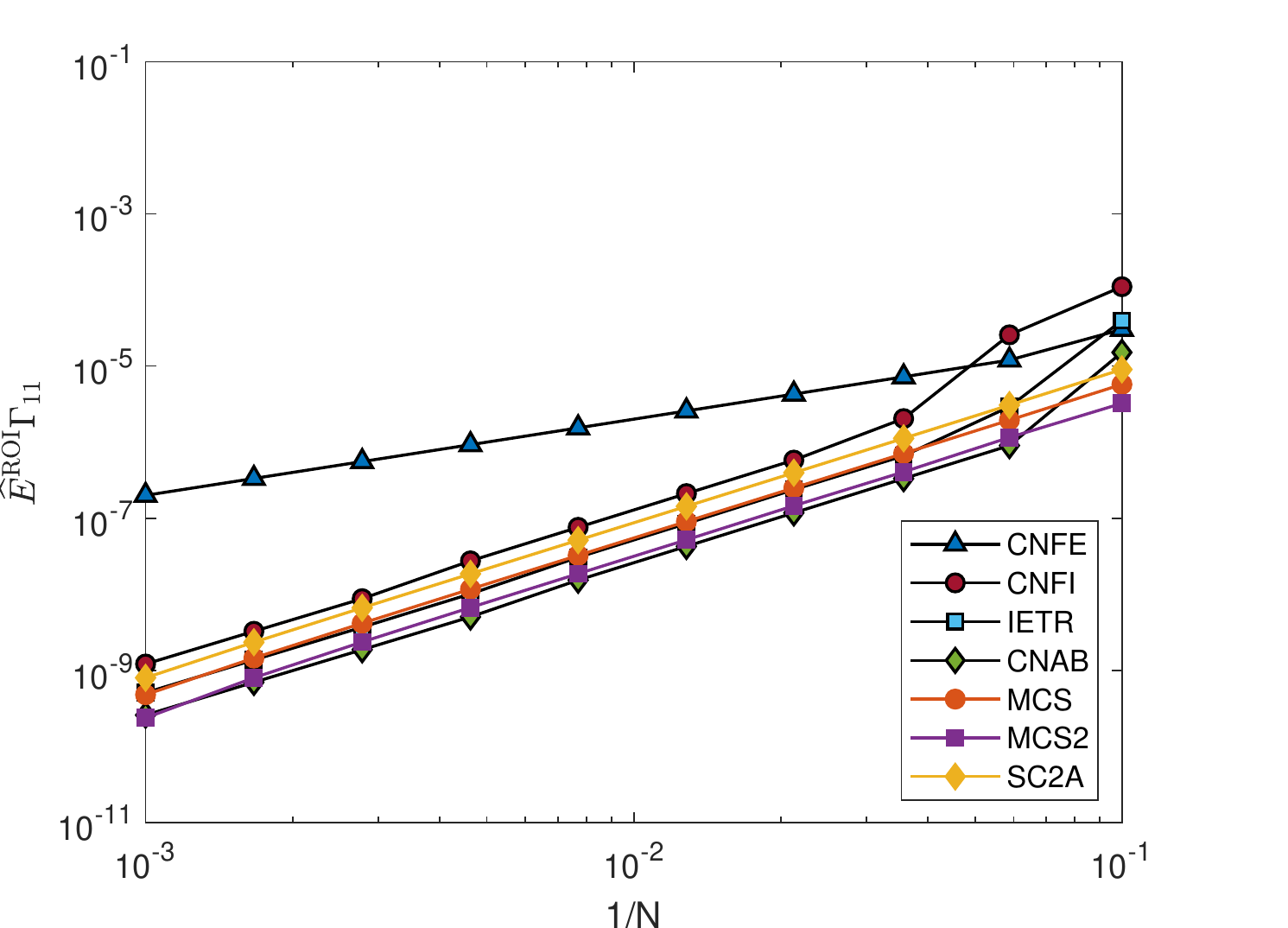}\includegraphics[trim = 1.3in 3.7in 1.6in 3.7in, clip, scale=0.5]{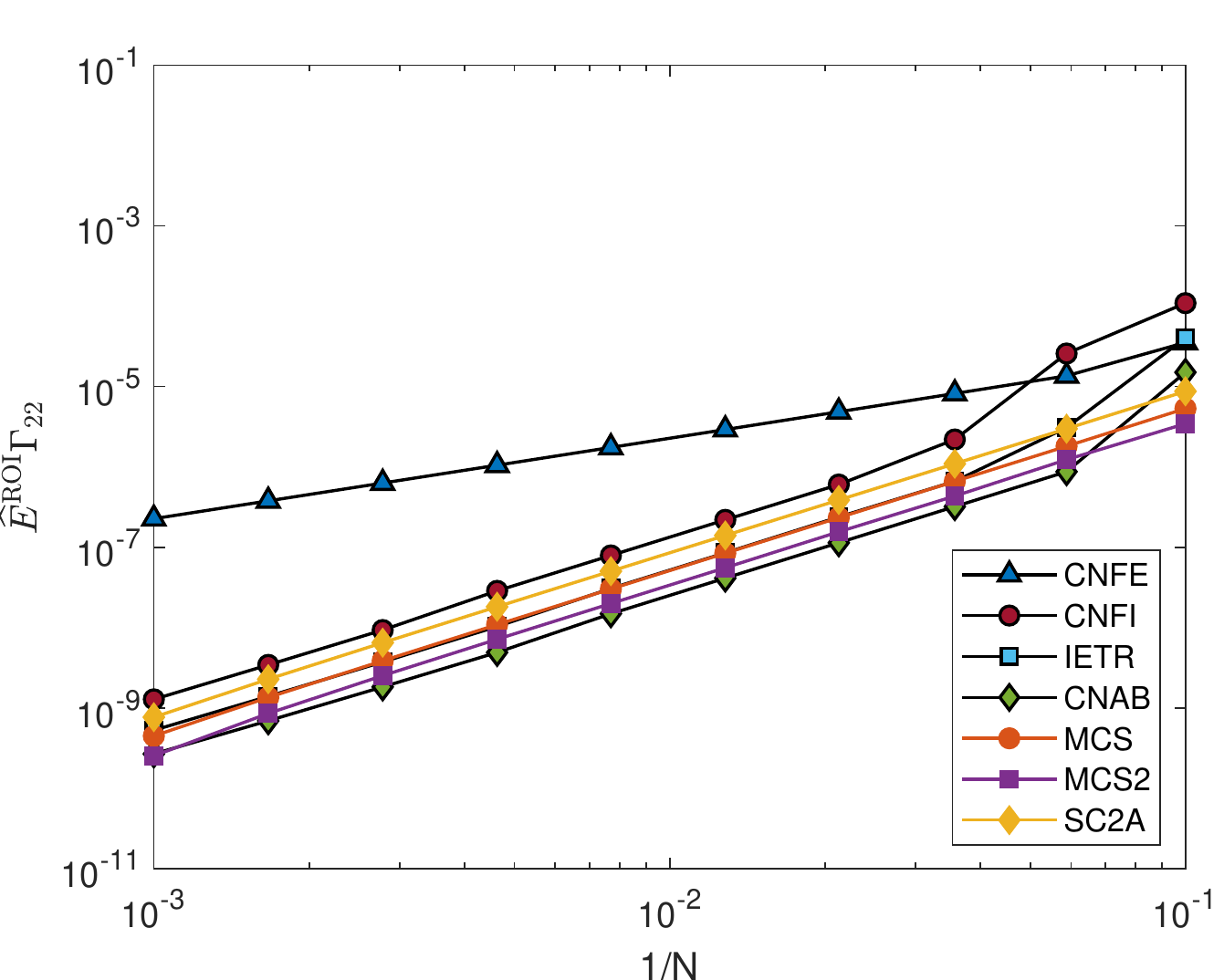}\\
	\includegraphics[trim = 1.3in 3.7in 1in 3.7in, clip, scale=0.5]{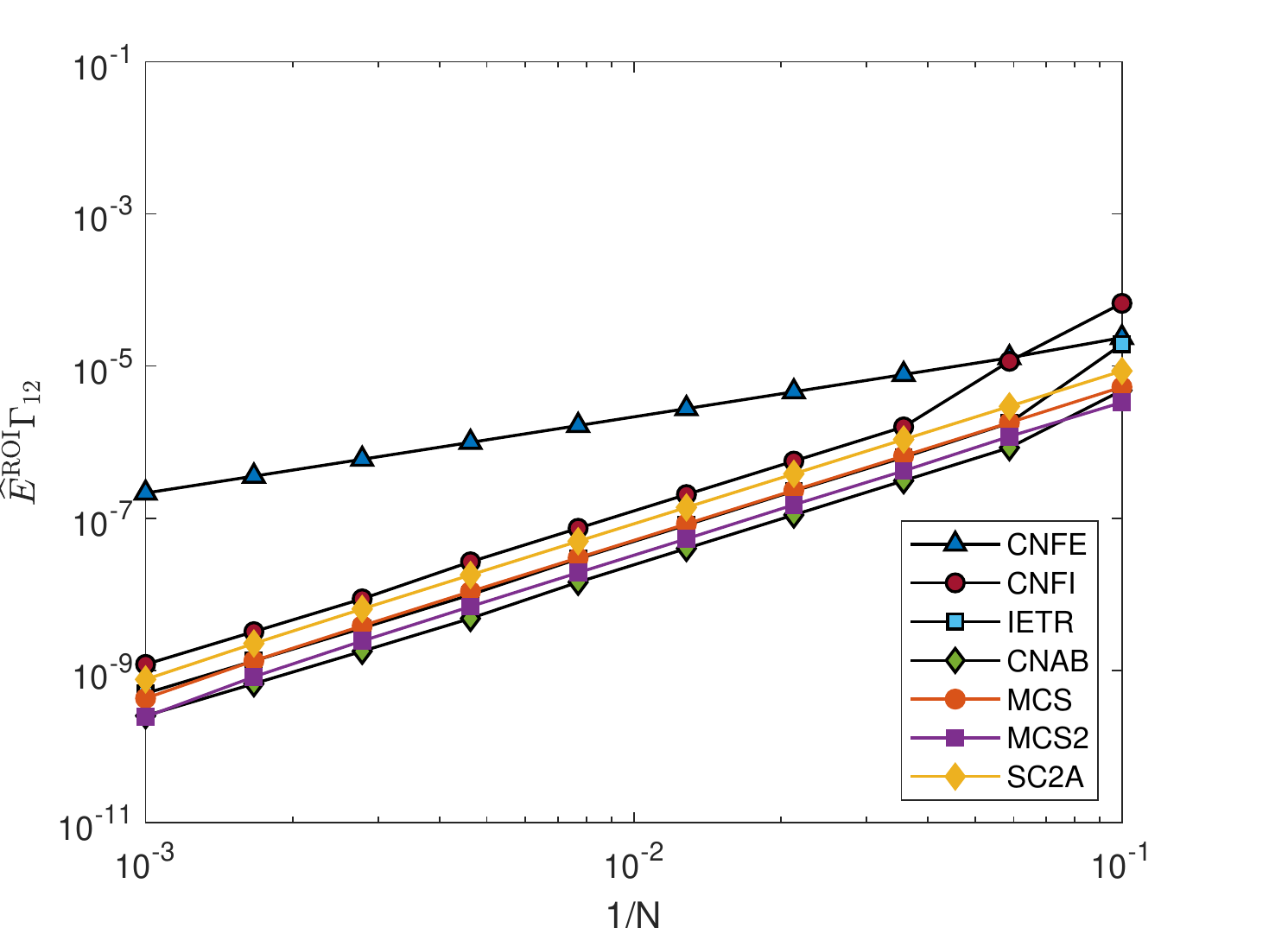}
	\caption{Temporal errors of the seven operator splitting schemes under consideration in the case of the five Greeks and parameter set 1: 
	$\widehat{E}^{\rm ROI}_{\Delta_1}$ (top left), $\widehat{E}^{\rm ROI}_{\Delta_2}$ (top right), $\widehat{E}^{\rm ROI}_{\Gamma_{11}}$ 
	(middle left), $\widehat{E}^{\rm ROI}_{\Gamma_{22}}$ (middle right) and $\widehat{E}^{\rm ROI}_{\Gamma_{12}}$ (bottom). 
	The CNFI and MCS schemes are applied with step size $\Delta t = T/N$, the IETR and MCS2 schemes are applied with 
	$\Delta t = T/\lfloor 3N/2 \rfloor$ and the CNFE, CNAB, SC2A schemes are applied with $\Delta t = T/(2N)$.}
	\label{FigGreeksError}
\end{figure}


\section{Conclusions}\label{SecConc}
We have studied the valuation of European options under the two-asset Kou jump-diffusion model via the numerical solution of the 
pertinent two-dimensional time-dependent PIDE.
A first main contribution of our paper is the extension of an algorithm derived by Toivanen~\cite{T08}, which enables a highly 
efficient numerical evaluation of the nonlocal double integral appearing in this PIDE.
The computational cost of the acquired algorithm is optimal: it is directly proportional to the number of grid points in the 
spatial discretization.
Also, it is simple to implement and requires little memory usage.
Subsequently, for the efficient discretization in time of the semidiscretized two-dimensional Kou PIDE, we have investigated 
seven modern operator splitting schemes of the IMEX and the ADI kind.
Every splitting scheme conveniently treats the integral term in an explicit fashion. 
Through rigorous analysis and extensive numerical experiments, we have examined the stability and convergence behavior, 
respectively, of the splitting schemes as well as their relative performance. 
All of the considered schemes, except for the first-order CNFE scheme, show a desirable, stiff order of temporal convergence 
equal to two. 
The MCS2 scheme, successively developed by in 't Hout \& Welfert \cite{HW09} for PDEs and in 't Hout \& Toivanen \cite{HT18} 
for PIDEs, stood out favorably among the splitting schemes in view of its superior efficiency. 
This conclusion agrees with the results recently obtained by Boen \& in 't Hout~\cite{BH21} in the case of the two-dimensional 
Merton PIDE. 

All schemes and results in our present paper can straightforwardly be extended to the case of a two-asset jump-diffusion model 
that has a mixture of independent and perfectly correlated jumps with log-double-exponential distributions, 
leading to a sum of three integrals in the two-dimensional PIDE, which is handled as a single (integral) term in the IMEX and 
ADI schemes, cf.~also Kaushansky, Lipton \& Reisinger~\cite{KLR18} in a different context.

For the valuation of American-style options under a given two-asset jump-diffusion model, a two-dimensional partial 
integro-differential complementarity problem (PIDCP) is obtained.
The adaptation of the operator splitting schemes of our present paper to such problems has recently been studied in 
Boen \& in 't Hout~\cite{BH20} in the case of the two-asset Merton model.
Here, for their effective adaptation, the combination with an iterated version of the Ikonen--Toivanen (IT) splitting 
technique~\cite{IT04,IT09,T08} has been considered as well as the penalty approach~\cite{CF08,FV02,ZFV98}.
Notably, the MCS2-IT(2) method~\cite{BH20}, which denotes the combination of the MCS2 scheme with two iterations of IT 
splitting, is found to be an efficient and stable temporal discretization method for the two-dimensional Merton PIDCP.
We expect that the same conclusion will hold in the case of the two-dimensional Kou PIDCP.

For PIDEs or PIDCPs that stem from infinite activity processes, such as the VG, NIG and CGMY models, the development and
analysis of operator splitting methods is still largely open in the literature and this forms an aim for future research. 
Here a useful idea is to replace the small jumps, that is, the jumps with sizes not exceeding a given small positive 
threshold $\varepsilon$, by a scaled Brownian motion, see~\cite{CT04book,CV05}.

\section*{Declaration of interest}
The authors report no conflicts of interest. The authors alone are responsible for the content and writing of the paper.

\section*{Acknowledgements}
The authors wish to thank the two anonymous reviewers for their various useful comments and suggestions, which have led
to a substantial improvement of the original version of this paper.


\bibliographystyle{plain}
\bibliography{bib2DKou_splitting}

\end{document}